\documentclass[a4paper,reqno, 11pt]{amsart}
\usepackage{graphicx,color,marginnote}
\usepackage{footmisc}
\usepackage{amsmath}
\usepackage[utf8x]{inputenc}

\usepackage{amsfonts}
\usepackage{amssymb}
\usepackage{bbm}
\usepackage{epstopdf}
\usepackage{xcolor}
\numberwithin{equation}{section}
\usepackage[marginparwidth=3cm, marginparsep=.5cm]{geometry}

\usepackage{hyperref}
\hypersetup{
    colorlinks=true,
    linkcolor=blue,
    citecolor=red,
    urlcolor=blue,
    pdfborder={0 0 0}
}

\newcommand{\R}{\mathbb{R}}

\newcommand{\Z}{\mathbb{Z}}

\newcommand{\E}{\mathbb{E}}

\renewcommand{\P}{\mathbb{P}}

\newcommand{\norm}[1]{\lVert #1 \rVert}
\newcommand{\abs}[1]{\lvert #1 \rvert}

\newcommand{\eps}{\epsilon}

\usepackage{cleveref}
  \crefname{theorem}{Theorem}{Theorems}
  \crefname{lemma}{Lemma}{Lemmas}
  \crefname{remark}{Remark}{Remarks}
  \crefname{proposition}{Proposition}{Propositions}
\crefname{notation}{Notation}{Notations}
\crefname{claim}{Claim}{Claims}
  \crefname{definition}{Definition}{Definitions}
  \crefname{corollary}{Corollary}{Corollaries}
  \crefname{section}{Section}{Sections}
  \crefname{figure}{Figure}{Figures}
    \crefname{assumption}{Assumption}{Assumptions}

\usepackage{amsthm}
\newtheorem{theorem}{Theorem}[section]

\newtheorem{remark}[theorem]{Remark}
\newtheorem{corollary}[theorem]{Corollary}

\newtheorem{assumption}[theorem]{Assumption}

\newtheorem{lemma}[theorem]{Lemma}
\newtheorem{definition}[theorem]{Definition}
\newtheorem{proposition}[theorem]{Proposition}

\renewcommand{\d}{{\sharp }}

\title{Note on mixing time in ``general'' non flat domains}

\renewcommand{\[}{\begin{equation}}
\renewcommand{\]}{\end{equation}}

\title{The mixing time of the  lozenge tiling Glauber dynamics}
\author{Beno\^it Laslier}
\address{Université Paris Cité,
UFR de Mathématiques, bâtiment Sophie Germain,
8 place Aurélie Nemour,
75205 Paris CEDEX 13, FRANCE} \email{laslier@lpms.paris}
\author{Fabio Lucio Toninelli}
\address{Technical University of Vienna, Institut f\"ur Stochastik und Wirtschaftsmathematik, Wiedner Hauptstra{\ss}e 8-10, A-1040 Vienna, Austria} \email{fabio.toninelli@tuwien.ac.at}
\begin{document}

\maketitle

\begin{abstract}
The broad motivation of this work is a rigorous understanding of
reversible, local Markov dynamics of interfaces, and in particular their
speed of convergence to equilibrium,  measured via the
mixing time $T_{mix}$.  In the $(d+1)$-dimensional setting, $d\ge2$, this is to a large extent
mathematically unexplored territory,
especially for discrete interfaces. On the other hand, on the basis of a
mean-curvature motion heuristics \cite{Spohn,Henley} and simulations (see
\cite{Destainville} and the references in \cite{Wilson,Henley}), one expects
convergence to equilibrium to occur on time-scales of order
$\approx \delta^{-2}$ in any dimension, with $\delta\to0$ the lattice mesh.
  
We study the single-flip Glauber dynamics for lozenge tilings of
a finite domain of the plane, viewed as $(2+1)$-dimensional surfaces.
The stationary measure is the uniform measure on admissible
tilings. At equilibrium, by the limit shape theorem \cite{CKP}, the
height function concentrates as $\delta\to0$ around a deterministic
profile $\phi$,  the unique minimizer of a surface tension
functional.  Despite some partial mathematical
results \cite{Wilson,LT_ptrf,LT_cmp}, the
conjecture  $T_{mix}=\delta^{-2+o(1)}$  has been
proven, so far, only in the situation where $\phi$ is an affine
function \cite{CMT}.  In this work, we prove the conjecture
 under the sole assumption that the limit
shape $\phi$ contains no frozen regions (facets).
\end{abstract}

\section{Introduction}

To define the problem that we study in this work, we start from
$\mathcal T_\delta$, the periodic triangular planar lattice where each
face is an equilateral triangle of side $\delta$. A \emph{tile} will
denote the lozenge-shaped polygon obtained by the union of two
adjacent triangular faces. Note that tiles can have three different
orientations, see Figure \ref{fig:1}.  We say that a domain $D_\delta$
(a bounded, connected union of faces of $\mathcal T_\delta$) is \emph{tilable}
if it can be covered with tiles, in such a way that tiles do not
intersect (except along their boundaries) and leave no 
hole. Call $\Omega_{D_\delta}$ the set of possible tilings of
$D_\delta$.
\begin{figure}[h]
\centering
  \includegraphics[scale=.25]{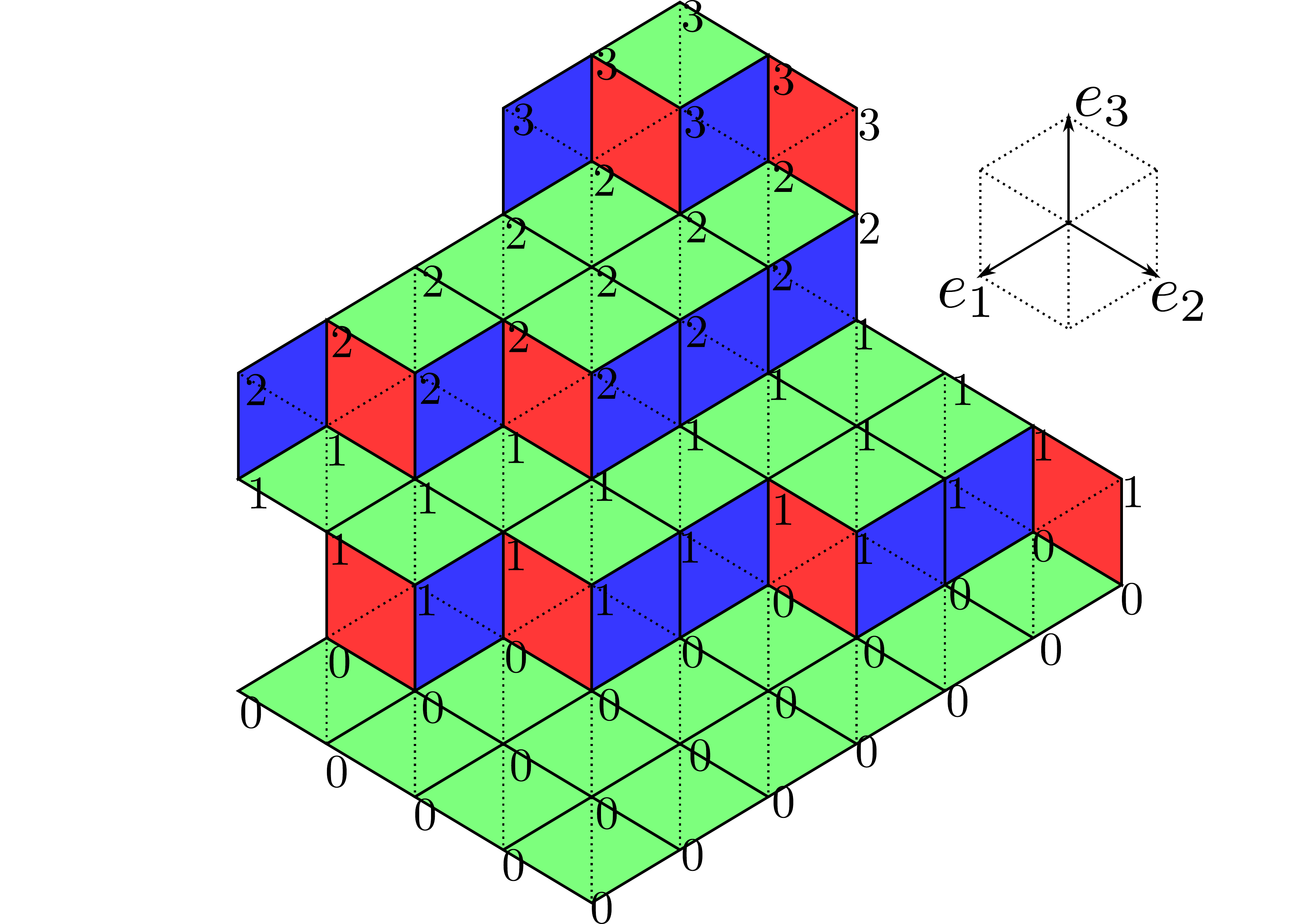}
  \caption{A tilable domain and the height function of a possible tiling. The value of the height function is given in units of $\delta$.}
\label{fig:1}
\end{figure}

A very natural, local, continuous-time Markov dynamics on the state space
$\Omega_{D_\delta}$ is defined by assigning a transition rate $1/2$ to
the two elementary updates in Fig. \ref{fig:2}, which consist in
rotating three lozenges by $120^\circ$ around their common vertex.
\begin{figure}[h]
\centering
  \includegraphics[scale=.3]{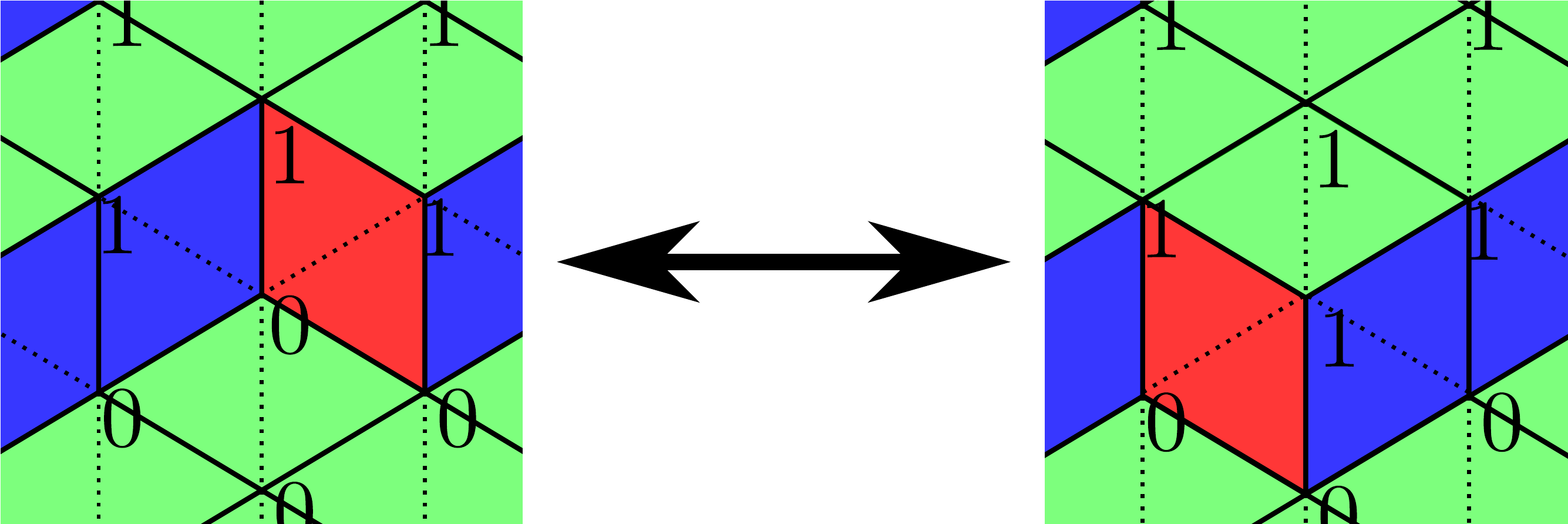}
  \caption{The elementary updates of the Glauber dynamics, each having transition rate $1/2$.}
\label{fig:2}
\end{figure}
We refer to this as the ``tiling Glauber dynamics'' and in fact, as is
well known (see e.g. \cite{CMT}), it can be seen as a zero-temperature
limit of the Glauber dynamics of the three-dimensional Ising model
with certain Dobrushin-type boundary conditions.  The stationary (and
reversible) measure of the process is the uniform distribution on
$\Omega_{D_\delta}$. The process is easily seen to be ergodic, and a
classical question in the domain of Markov chains \cite{LP} is to
understand how fast equilibrium is reached, as measured via the mixing
time $T_{mix}$.  In \cite{LRS} Luby, Randall and Sinclair proved that
if $D_\delta$ has diameter of order $1$, so that it contains
approximately $\delta^{-2}$ faces, then in the limit $\delta\to0$ the
mixing time grows at most like a polynomial: $T_{mix}\lesssim
\delta^{-C}$ for some $C>0$. This ``fast mixing'' result, based on a
smart coupling idea, is already non-trivial in view of the fact that
for most ``reasonable'' domains $D_\delta$, the state space is of
cardinality $\approx \exp(c/\delta^{2})$. The main result of the
present work is that, for a class of natural domains $D_\delta$ that
approach as $\delta\to0$ some bounded domain $D\subset \R^2$, the
mixing time is $T_{mix}=\delta^{-2+o(1)}$.

In order to better explain our result, the expected picture and the
motivations for our work, let us take a step back and put the problem
into a more general context. The first important observation is that a
lozenge tiling can be mapped into a discrete height function $h$,
defined on the collection $W_\delta$ of the tile vertices, and taking
values in $\delta\Z$ (see Fig. \ref{fig:1}). The tiling Glauber
dynamics can then be seen as a continuous-time, reversible, Markov
evolution of a discrete height function with local update rules, where
elementary updates consist in adding $\pm \delta$ to the height of a
single vertex. The equilibrium measure, that we call
$\pi_{W_\delta,h|_{\partial W_\delta}}$, is the uniform measure on
height functions satisfying a certain local Lipschitz constraint (see
\eqref{eq:gradients}) and with fixed boundary value $h|_{\partial
  W_\delta}$ on the boundary of $W_\delta$, determined by the shape of
the domain $D_\delta$.  In the scaling limit where the lattice mesh
$\delta$ tends to $0$, assuming that the domain $W_\delta$ tends to a
continuous bounded domain $W\subset \R^2$ and the boundary height
tends to some Lipshitz function $\phi|_{\partial W}$, the height
function sampled from the measure $\pi_{W_\delta,h|_{\partial
    W_\delta}}$ tends in probability to a deterministic \emph{limit
  shape} $\phi$, that minimizes a surface energy functional
\cite{CKP,CLP}. Interestingly, for certain boundary conditions the
limit shape has facets (or ``frozen regions'') which, at the
microscopic level, contain with overwhelming probability only one of
the three types of tiles. The phenomenon of appearance of facets in
the limit shape is called ``arctic circle'' or ``frozen boundary''
phenomenon \cite{CKP,CLP,KO}.

On the basis of phenomenological arguments, one expects a hydrodynamic
limit on the diffusive time scale $1/\delta^2$, i.e., the height profile $h_{t/\delta^2}(\cdot)$
should tend as $\delta\to0$ to the solution $H(t,\cdot)$ of a parabolic
PDE of the type
\[
\label{eq:hydro}
\partial_t H=\mathcal L H,
\]
with $\mathcal L$ an elliptic non-linear operator such that $\mathcal
L\phi=0$ if $\phi$ is the limit shape.  Note that, since $\mathcal L$
is elliptic, \eqref{eq:hydro} is a variant of mean curvature motion; the main difference is that its stationary points minimize surface tension instead of surface area.  On the basis
of this picture, it is natural to conjecture that the time-scale for
equilibration (mixing time) is of order $\delta^{-2+o(1)}$.  The
$o(1)$ hides unavoidable sub-dominant corrections: in fact, we
point out that even in the much easier case of $1-$dimensional symmetric simple
exclusion process (SSEP) \cite{Wilson,Lacoin} or one-dimensional $\nabla\varphi$ interface dynamics with convex potential \cite{CLL}, the mixing time turns out to be of order
$\delta^{-2}\times |\log\delta|$.

The coupling argument of \cite{LRS} does not at all use this ``mean
curvature evolution'' intuition and, not surprisingly, it does not
allow to capture the conjectured exponent $C=2$ in the
$\delta^{-C+o(1)}$ behavior of the mixing time. The first step towards
establishing the conjecture was performed in \cite{CMT}, where it was
proven that indeed $T_{mix}=\delta^{-2+o(1)}$ under the important
restriction that the limit shape $\phi$ is an affine function (this is a non-trivial restriction on the domain $D_\delta$). The
simplifying feature of this case, in few words, is that the more
the interfaces approaches equilibrium, the more its law resembles
(locally) that of an infinite-volume, translation invariant Gibbs
state, for which very sharp height fluctuations results are known
\cite{Kenyon-lectures} (the role of such estimates will become clear later in the paper).

In the generic case where the limit shape $\phi$ is curved, the only
available mathematical confirmation of the expected $\delta^{-2+o(1)}$
time-scale for equilibration is provided by \cite{LT_cmp}, which
proves that at such time scales the height profile is with high
probability within $L^\infty$ distance $\epsilon$ from $\phi$, for any
fixed $\epsilon>0$. We emphasize that this result, while suggestive of
the expected behavior, has no implications on $T_{mix}$: it could very
well be that the height profile is at some time at very small $L^\infty$ distance
from the limit shape, yet the \emph{law} of the height function at
the same time is at variation distance essentially $1$ from the
equilibrium measure.

In the present work, as already mentioned, we prove the
$T_{mix}=\delta^{-2+o(1)}$ conjecture, under the sole assumption that
$\phi$ contains no facets. Our method builds on one hand on an
iterative procedure first developed in \cite{CMT} (which itself is
inspired by the mean-curvature heuristics), and on the other hand on
sharp equilibrium fluctuation results on domains of mesoscopic size
(much larger than the mesh size but much smaller than $1$).  Let us
emphasize that, while for instance the recent \cite{Aggarwal} provides a
``local equilibrium'' type of result on microscopic domains of order
$\delta$ (where the equilibrium measure coincides asymptotically with
a translation-invariant Gibbs measure), we really need to consider
mesoscopic domains where the effect of the curvature of the height
profile, that drives interface motion, is non-negligible. Technically,
we rely in this respect on the recent work \cite{L_arxiv}.

\subsection{The broader context}

To put our result in a broader context, let us point out that the
$\delta^{-2+o(1)}$ behavior of the mixing time is expected for many
natural, local, reversible interface dynamics in any dimension
$(d+1),d\ge1$, and not just when $d=2$ as in the case under
consideration. The result, in its sharper form $T_{mix}\approx \delta^{-2}|\log \delta|$,  is known to hold in several
$(1+1)$-dimensional reversible interface dynamics, most notably, the
symmetric simple exclusion (SSEP) (\cite{Wilson,Lacoin}, see also
\cite{KL} for the hydrodynamic limit) and Ginzburg-Landau
$\nabla\varphi$ models with convex potential (see \cite{CLL}, where a
proof of the total variation cut-off phenomenon is also obtained). It is important to
emphasize that the $(1+1)$-dimensional case is easier for at least two
reasons. First, by looking at interface gradients, the dynamics can be
seen as an interacting particle system, whose equilibrium distributions are
often of i.i.d. type (for instance, they are i.i.d. Bernoulli measures for
SSEP), while interface gradients exhibit power-law decaying
correlations in higher dimension. Secondly, limit shapes are
 affine in one dimension (by convexity of the surface
tension), while they are generically curved in higher dimension.

On the other hand, we are not aware of a result comparable to ours in
dimension $(d+1),d\ge2$, especially for discrete interfaces. A notable
exception is \cite{Wilson}, which proves $\delta^{-2}|\log \delta|$
bounds for the mixing time of a \emph{non-local} lozenge tiling
dynamics, where ``non-local'' refers to the fact that each update
flips a random number, that can be as large as $\delta^{-1}$, of
tiles. Let us also mention that for continous-height, \emph{Gaussian}
interface dynamics (that is, Ginzburg-Landau $\nabla \phi$ models with quadratic
interaction potential, also known as the discrete GFF) the proof of
$T_{mix}\approx \delta^{-2}|\log \delta|$ is rather easy, via an
application of the method of \cite{Wilson}. In the $(2+1)$-dimensional
case, for this Gaussian process, the total variation cut-off
phenomenon has also been proven to hold \cite{GG}.

A final remark: \emph{biased} versions of (discrete) interface
dynamics have been considered in the literature. In this case, the
rate for updates that increase the height is $p$ and that of updates
that decrease it is $q$ with, say, $p<q$. In infinite volume, these
processes are irreversible and celebrated examples are the
one-dimensional Asymmetric Simple Exclusion Process (ASEP) and its
totally asymmetric counterpart, the TASEP (corresponding to $p=0$). In
finite volume, instead, these dynamics are reversible, and the
stationary measures are just those of the symmetric processes, tilted
by $(p/q)^V$, with $V$ the volume below the interface. In this case,
the phenomenology is \emph{qualitatively very different} from the one studied
in the present work.  For one thing, macroscopic shapes in this case
minimize not the free energy but the free energy with a volume
constraint, and the mixing time turns out to scale like
$\delta^{-1+o(1)}$ rather than $\delta^{-2+o(1)}$. We refer to
\cite{LLabbe} for the ASEP in an interval (the authors prove the sharp
estimate $T_{mix}\sim c \delta^{-1}$, as well as the occurrence of the
cut-off phenomenon) and to \cite{GP,CMTRSA} for a dynamics of biased
plane partitions, which is a lozenge tiling dynamics where the two
updates of Fig. \ref{fig:2} have different transition rates $p$ and
$q$. Also in the latter case, the result is that
$T_{mix}=\delta^{-1+o(1)}$ (see \cite{GP}, where this is proven for small enough bias $\log(p/q)$, and \cite{CMTRSA} for the general case of arbitrary non-zero bias).

\subsection*{Organization of the article} The rest of this work is organized as follows. In Section
\ref{sec:results} we define precisely the problem and state the main
results.  In Section \ref{sec:prelsketch}, we give some preliminary
results and we present a sketch of the strategy of the proof of our
main theorem. The proof of the mixing time bound is reduced in Section
\ref{sec:UB} to an inductive statement, \cref{th:core}. In Section
\ref{sec:core} we prove \cref{th:core} up to two sharp equilibrium
results on mesoscopic scales, whose proofs are the contents of the
final \cref{sec:perturbation_continuum}.

\section{Statement of the problem and results}
\label{sec:results}
The Glauber dynamics on lozenge tilings can be seen as a
continuous-time Markov process on a certain set of discrete-valued
height functions, and this is the point of view we adopt in the whole
paper.  The height function is defined on a portion of the planar
triangular lattice. Since we are interested in large-scale behavior,
we rescale the lattice mesh by $\delta>0$ and we denote $\mathcal
T_\delta$ the rescaled lattice.  We let $e_1,e_2,e_3$ denote the
elementary vectors of Figure \ref{fig:1}.

We start with a few preliminary definitions.
\subsection{Preliminary definitions}
\label{sec:somenot}
Given an open bounded domain $D$ of $\mathbb R^2$,  a continuous function $f:D\mapsto \mathbb R$ and a real number $\delta>0$, we let:
\begin{itemize}
\item
$\partial D$ be the boundary of $D$;
\item $D^{\sharp}$ be the
collection of vertices of $\mathcal T_\delta$ that are contained in
$D$;
\item for $V\subset D$, $f|_V$ be the restriction of $f$ to $V$.
\end{itemize}

Given $K>0$, we  say that $f^{\d}:D^{\d}\mapsto \delta \Z$ is a $K$-discretization of $f:D\mapsto \R$ if 
\begin{eqnarray}
  \max_{x\in D^\d}|f^\d(x)-f(x)|\le K\delta.
\end{eqnarray}
Finally, given a subset $A$ of $\mathcal T_\delta$ (for instance, $D^\d$)
we let $\partial A$ be the collection of vertices in $A$ that have
  a neighbor in $\mathcal T_\delta\setminus A$.

  Given $A\subset \mathcal T_\delta$, a function $h:A\mapsto \delta \Z$ is said to be an admissible height  function (or just ``height function'' for short) on $A$ if whenever $x,y\in A$ are nearest neighbors, then
  \begin{equation} 
  h(x)-h(y) \in
    \begin{cases}
       \{-\delta, 0\} &\text{ if }  x-y=e_1 \text{ or } x-y=e_2\\
        \{0, \delta\} & \text{ if }  x-y=e_3. 
    \end{cases} \label{eq:gradients}
  \end{equation}
See Fig. \ref{fig:1}. 

\begin{remark}[The Newton Polygon $\mathbb T$]Because of the conditions \eqref{eq:gradients} imposed on the discrete gradients of $h$, not any continuous function $f:\R^2\mapsto \R$ admits a discretization
that is an admissible height function on $\mathcal T_\delta$. For that
to be possible, $f$ must be Lipchitz and, in addition, $\nabla f$ must
belong to a certain polygon, called ``Newton Polygon''. If we choose a coordinate
system on $\R^2$ with coordinate axes parallel to the vectors
$e_1,e_2$ of Fig. \ref{fig:1}, then $\mathbb T$ is the triangle of
vertices $(0,0),(-1,0),(0,-1)$.  If the function $f:D\mapsto \R$ is
Lipschitz and $\nabla f\in \mathbb T$, then 
there exists a $1-$discretization $f^\d$ that is a height function on $D^\d$ (see e.g. \cite[Sec. 2.1]{LT_cmp}
and \cite[Sec. 2.5.4]{LT_ptrf}).
\label{rem:T}
\end{remark}

  Given $A\subset \mathcal T_\delta$ and a boundary height function $g:\partial A\mapsto \delta\Z$, a
  height function $h:A\mapsto \delta \Z$ is said to be compatible with
  the boundary value if it coincides with $g$ on $\partial A$.  We let $\Omega_{A,g}$ be the (possibly empty) set of
  height functions on $A$, compatible with $g$.

  \subsection{The dynamics, its stationary measure and the limit shape phenomenon}

Given $A\subset \mathcal T_\delta$ and a boundary height function
$g:\partial A\mapsto \delta\Z$, we define the dynamics as a
continuous-time Markov process on $\Omega_{A,g}$, assumed to be
non-empty. This can be described as follows: each vertex in
$A\setminus \partial A$ has an i.i.d. mean-one Poisson clock. When the
clock at $x$ rings, the height function at $x$ is resampled uniformly
among the possible values of $h(x)$ given $\{h(y),y\ne x\}$. Note
that, because of \eqref{eq:gradients}, the number of such possible
values is either $1$ (in which case we can equivalently say that the
update does not occur) or $2$.  It is also elementary to see that the
uniform measure $\pi_{A,g}$ on $\Omega_{A,g}$ is stationary and
reversible.

Assume now that the domain $A$ is the discretization $D^\d$ of an open
domain of $\mathbb R^2$ and the boundary height $g$ is the restriction
to $\partial D^\d$ of the discretization $f^\d$ of some $f:\R^2\mapsto
\R$, such that $\Omega_{D^\d,f^\d|_{\partial
    D^\d}}\ne\emptyset$.  Then, the height
function exhibits a \emph{Law of Large Numbers}  \cite{CKP,CLP}, better known as \emph{Limit Shape Theorem} in this context: there exists a
unique,  continuous, deterministic function $\phi:D\mapsto \R$ such that, for every $\epsilon>0$,
\begin{eqnarray}
  \label{Eq:limshape}
\pi_{D^\d,f^\d|_{\partial D^\d}}\big(\|h-\phi\|_\infty\ge \epsilon\big)\stackrel{\delta\to0}\longrightarrow0,
\end{eqnarray}
where $\|h-\phi\|_\infty:=\max_{x\in D^\d}|h(x)-\phi(x)|$ and
$\phi$ satisfies the boundary condition
$\phi|_{\partial D}=f|_{\partial D}$.
The function $\phi$ is called \emph{limit shape} and it is the unique minimizer of a \emph{surface tension functional}
\[
\Phi(\phi)=\int_D \sigma(\nabla \phi)dx
\]
among all functions with Dirichlet boundary conditions $\phi|_{\partial D}=f|_{\partial D}$. The function $\sigma$ is fully explicit \cite{Kenyon-lectures} but its specific form is not crucial for the present work. In the following remark we summarize all what we need to know about it.

\begin{remark}[Surface tension and Euler-Lagrange equation]
  The function $\sigma$ is convex; it is finite in the Newton Polygon $\mathbb T$, equals zero on its boundary $\partial \mathbb T$ and equals $+\infty$ outside of it. In the interior of $\mathbb T$, $\sigma$ is real analytic and strictly convex. More precisely, $\sigma_{ii}>0$, where $\sigma_{ij},1\le i,j\le 2$ denotes the derivative of $\sigma$ with respect to its $i^{th}$ and $j^{th}$ arguments.
 In the regions of $D$ where the limit shape $\phi$ is $C^2$, it satisfies the Euler-Lagrange equation
\begin{eqnarray}
  \label{eq:EL}
  {\rm div} (\nabla\sigma\circ\nabla \phi):=\sum_{i,j=1}^2 \sigma_{ij}(\nabla\phi)\partial^2_{x_i x_j}\phi= 0,
\end{eqnarray}
which is a non-linear elliptic PDE.
\end{remark}
A very interesting and well-known aspect of the limit shape $\phi$ of
lozenge tilings is that it can exhibits facets. That is, there are
domains $D$ and boundary conditions $f$ such that the minimizer $\phi$
contains both \emph{liquid regions} where $\nabla \phi$ belongs to the
interior of $\mathbb T$ and \eqref{eq:EL} holds, and also \emph{frozen
  regions} (with non-empty interior) where $\nabla\phi\in \partial
\mathbb T$. The boundary between liquid and frozen regions is usually
called \emph{frozen boundary} and it consists in algebraic curves \cite{KO,Duse}.

In this work, we study the mixing time of the Glauber dynamics, under
the assumption that the limit shape $\phi$ has no such frozen
regions. Understanding the interplay between dynamics and  frozen boundaries remains an interesting and  challenging open problem.

\subsection{The mixing time bounds}

In order to define the precise setting of our results, we start from 
a simply connected, bounded, open domain $U$ of
$\mathbb R^2$, whose boundary is a Jordan curve.
\begin{assumption}[The limit shape]
  \label{ass:ls}
We let $\phi:U\mapsto \mathbb R$ be a
$C^\infty$ function that satisfies the Euler-Lagrange equation
\eqref{eq:EL} in $U$. We further
require that $\phi$ contains no frozen regions or more precisely that
there exists $a>0$ such that
\begin{eqnarray}
  \label{eq:nonfroz}\inf_{x\in   U}d(\nabla\phi(x),\partial\mathbb T)\ge a,
\end{eqnarray}
  with $d$ the Euclidean distance.
\end{assumption}
Actually, because the surface tension $\sigma$ is
$C^\infty$ and real analytic in the interior of the Newton polygon, the weaker assumption
\begin{eqnarray}
  \label{eq:weaker}
  \phi\in C^1(U) \text{ and } \eqref{eq:nonfroz} \text{ holds }
\end{eqnarray} implies that $\phi$  is $C^\infty$ and real analytic in $U$ \cite[Sec. VI.3]{Giaquinta}. In particular,  in \cref{ass:ls} we can replace the $C^\infty$ by the $C^2$ requirement. We refer to \cite{Duse} for much deeper information on the regularity properties of limit shapes.

\begin{assumption}[The boundary condition on the dynamics]
  \label{assu:W}
We let $W$ be a simply connected, bounded, open domain of
$\mathbb R^2$, whose boundary is a Jordan curve, and such that
$\overline W\subset  U$. The dynamics is  defined in the
discrete set $W^\d$ and the boundary condition is given by
$\phi^\d|_{\partial W^\d}$, i.e., the restriction to the
boundary of $W^\d$ of some $K$-discretization of $\phi$ with  $\Omega_{W^\d,\phi^\d|_{\partial W^\d}}$ non-empty.
  
\end{assumption}
The constants in the following theorem can depend implicitly on $K,W,\phi$.
 \begin{theorem}
   \label{th:main}
   For the Glauber dynamics in  $W^\d$  with boundary condition  $\phi^\d|_{\partial W^\d}$ satisfying Assumption \ref{assu:W}, there exists a constant $c_->0$ and, for every $\eta>0$, a constant
   $c_+<\infty$ such that the mixing time satisfies for all $\delta>0$
  \begin{eqnarray}
    \label{eq:main}
c_- \delta^{-2}\le     T_{mix}\le c_+\delta^{-2-\eta}.
  \end{eqnarray}
\end{theorem}

Let us mention that  the result of Theorem
\ref{th:main} would hold (with minor changes in the proof) under the
weaker assumption that the boundary height $h|_{\partial W^\d}$ satisfies
\begin{eqnarray}
  \label{bordovicino}
\|h|_{\partial W^\d}-\phi|_{\partial W^\d}\|_\infty\stackrel {\delta\to0}=O(\delta^\eta)
\end{eqnarray}
 for every $\eta\in (0,1)$.  
In this case, the constant $c_+$ in the theorem would also depend on the constants implicit in the estimate \eqref{bordovicino}.

\begin{remark}
  The example one should keep in mind is that $U$ is the ``natural liquid region''
   for some limit shape $\phi$ (i.e. the maximal domain where $\phi$ is smooth \cite{KO}), and $W$ is obtained from $U$ by removing an
  $\varepsilon-$neighborhood of $\partial U$, with $\varepsilon>0$
  arbitrarily small but independent of the lattice mesh $\delta$.  One
  reason why we cannot just take $\varepsilon=0$ is that we will need some uniform
  control on the smoothness of $\phi$ in a neighborhood of  $\partial W$ and, in general, the derivatives of $\phi$ are singular at the natural boundary of liquid regions.
\end{remark}

\begin{remark}
  The novelty in Theorem \ref{th:main} is the upper bound. In fact,
  the argument for $T_{mix}\ge c_-\delta^{-2}$ given in
  \cite[Sec. 5.2]{LT_ptrf} for the case where the limit shape $\phi$
  is affine, works identically in the general case and we will
  not repeat it. We emphasize that, in contrast to the argument for
  the upper bound, the one for the lower bound is very soft and boils
  down to proving that at times $\delta^{-2}$ times a small constant,
  the height function is essentially unchanged (at the macroscopic
  scale) with respect to the initial condition.
\end{remark}

\section{Preliminaries and strategy of the proof}
\label{sec:prelsketch}
We start with a couple of useful general facts on continuous-time, irreducible Markov chains on a finite state space $\Omega$ (we refer the reader e.g. to \cite{LP}).
First of all, the mixing time is defined as
\begin{eqnarray}
  \label{eq:defTmix}
  T_{mix}:=\inf\{t\ge0: \max_{\eta\in \Omega}\|\mu_t^\eta-\pi\|\le 1/4\}
\end{eqnarray}
where $\mu_t^\eta$ denotes the law of the process at time $t$ with initial condition $\eta$ and $\pi$ is the unique stationary probability measure.
Here, $\|\cdot-\cdot\|$ denotes the total variation distance between probability measures:
\begin{eqnarray}
\|\mu-\nu\|=\max_{B\subset \Omega}|\mu(B)-\nu(B)|=\max_{B\subset \Omega}\mu(B)-\nu(B).
\end{eqnarray}
We will use the following standard sub-multiplicativity property of
total variation:
\begin{eqnarray}
  \label{eq:submult}
\max_{\eta\in \Omega}\|\mu_t^\eta-\pi\|\le 2^{-\lfloor t/T_{mix}\rfloor}.
\end{eqnarray}
We will also need a version of the union-bound with respect to time, for continuous time Markov chains:
\begin{lemma}
  Given $\eta\in \Omega$, let $c(\eta)$ be the sum of the rates of the transitions
  outgoing $\eta$ and let $c^+:=\max\{c(\eta),\eta\in\Omega\}$. Then,
  for any $A\subset \Omega$ and $T>0$,
\begin{eqnarray}
  \label{eq:unionboundt}
\P_\pi(\exists t\le T: X_t\in A)\le 8Tc^+\pi(A)
\end{eqnarray}
  where $X_t$ denotes the state of the chain at time $t$ and $\P_\pi$ is the law of the stationary process (i.e., with initial condition distributed according to $\pi$).
\end{lemma}
\begin{proof}
  Let $\tau$ be the hitting time of $A$, so that the l.h.s. of
  \eqref{eq:unionboundt} is $\P_\pi(\tau\le T)$, and let $L_A(t)$ be
  the total time spent in $A$ up to time $t$.  We have
  \begin{eqnarray}
    \label{eq:mw}
\P_\pi(\tau\le T)=\P_\pi(\tau\le T;L_A(2T)\le 1/(2c^+))+\P_\pi(\tau\le T;L_A(2T)> 1/(2c^+)).
  \end{eqnarray}
  For the first term in the r.h.s. we note that the time spent in $A$
  just after $\tau$ is stochastically dominated below by an
  exponential random variable $Y$ of parameter $c^+$ (this is a lower
  bound on the time it takes before the next update
  occurs). Therefore, by the strong Markov property that probability
  is upper bounded by
  \[
\P_\pi(\tau\le T)\P(Y<1/(2c^+))\le (1/2)\P_\pi(\tau\le T).
\]
As for the second term in the r.h.s. of \eqref{eq:mw}, we upper bound it by
\begin{multline}
  \P_\pi(L_A(2T)> 1/(2c^+))= \P_\pi\Big(\int_0^{2T}{\bf 1}_{X_t\in A}dt>1/(2c^+)
  \Big)\\
  \le 2c^+\E_\pi\Big(\int_0^{2T}{\bf 1}_{X_t\in A}dt\Big)=4 c^+T\pi(A)
\end{multline}
where we used Markov's inequality and stationarity of $\pi$. Putting everything together, \eqref{eq:unionboundt} follows.
\end{proof}
In the case of the Glauber dynamics on height functions that is the focus of the present work, we will use the following immediate corollary:
\begin{corollary}
  \label{cor:ubt}
  Consider the Glauber dynamics $\{h_t\}_{t\ge0}$ in a domain
  $D^\d\subset \mathcal T_\delta$ of cardinality $|D^\d|$ and with any
  boundary condition $h|_{\partial D}$. Given a set
  $A\subset \Omega_{D^\d,h|_{\partial D^\d}}$ of height functions, one
  has
  \begin{eqnarray}
    \label{eq:ubt}
\P_\pi(\exists t<T: h_t\in A)\le 8 |D^\d|T\pi(A),
  \end{eqnarray}
  with $\pi=\pi_{D^\d,h|_{\partial D^\d}}$.
\end{corollary}

A crucial tool is monotonicity.  Given two height functions $h,h'$ on
a domain $D^\d$, we say that $h\le h'$ iff $h(x)\le h'(x)$ for every
$x\in D^\d$. The Glauber dynamics is well-known (and easily checked)
to be monotone with respect to this partial order: we can couple all
the dynamics started from different initial conditions in a way that
if two initial conditions $h^{(1)},h^{(2)}$ satisfy $h^{(1)}\le
h^{(2)}$, then almost surely $h^{(1)}_t\le h^{(2)}_t$ for every
$t\ge0$, with obvious notations. Note that it is possible to have
$h^{(1)}_{\partial D^\d}\not\equiv h^{(2)}_{\partial D^\d}$, in which
case the two dynamics evolve with different boundary conditions.
Monotonicity is inherited by the equilibrium measure: if $g,g'$ are boundary heights on $\partial A$ and $g\le g'$ pointwise, then one has the stochastic domination $\pi_{A,g}\preceq \pi_{A,g'}$.

Along the proof of Theorem \ref{th:main}, we will need to consider ``constrained'' or ``censored'' versions of the dynamics, where the height function is constrained to satisfy $h^-(x)\le h_t(x)\le h^+(x)$ for all times $t$ in certain deterministic intervals, where $h^\pm$ are two fixed functions and ``constrained'' means that updates that violate these inequalities are discarded (censored).
A very useful auxiliary result is a bound on the mixing time for this constrained dynamics:
\begin{lemma}\cite[Th. 4.3]{CMT}
  \label{lemma:PTRF}
  The dynamics in a domain $D^\d$ constrained between $h^-$ and $h^+$ has a mixing time
\begin{equation}
  T_{mix}\le c\, {\rm diam}(D)^2\|h^+-h^-\|_\infty^2\delta^{-4}(\log\delta)^2
\end{equation}
  for some universal constant $c$, where ${\rm diam}(D)$ denotes the diameter of $D$.
\end{lemma}
\begin{remark}
Note that the equilibrium measure of the dynamics constrained between $h^-$ and $h^+$ is simply
\[
\pi_{W^\d,\phi^\d|_{\partial W^\d}}(\cdot|h^-\le h\le h^+).
  \]
\label{rem:eqmes}  
\end{remark}
From \cref{lemma:PTRF} it follows immediately:
\begin{corollary}\label{cor:L4}
  For the unconstrained Glauber dynamics in a domain $D^\d$, the mixing time satisfies
  \begin{eqnarray}
T_{mix}\le c\, {\rm diam}(D)^4\delta^{-4}(\log\delta)^2
  \end{eqnarray}
   for some universal constant $c$ that is independent of the boundary condition $h|_{\partial D^\d}$.
\end{corollary}
This follows simply by noting that the dynamics is trivially
constrained between the minimal and maximal configurations in
$\Omega_{D^\d,h|_{\partial D^\d}}$, whose height functions are
$1$-Lipschitz so that they are at sup-distance at most
${\rm diam}(D)$.

We also recall the main result of \cite{LT_cmp}:
\begin{theorem}\label{thm:CMPbound}
  For the Glauber dynamics with domain and boundary condition satisfying Assumption \ref{assu:W}, 
  for every $\epsilon,\eta>0$ there exists $c>0$ such that for every $\delta>0$ and  for every  $h\in\Omega_{W^\d,\phi^\d|_{\partial W^\d}}$ one has
  \begin{eqnarray}
    \label{eq:CMPbound}
    \mathbb P_{h}\Big(\exists t\in [(c/\delta)^{2+\eta},(1/\delta)^5] :\max_{x\in W^\d}|h_t(x)-\phi(x)|\ge \epsilon\Big)\le \epsilon,
  \end{eqnarray}
  where $\P_h$ denotes the law of the process started at $h$.
\end{theorem}

Actually, Theorem 3.1 of \cite{LT_cmp} provides a similar statement
but for any fixed $t\in [(c/\delta)^{2+\eta},(1/\delta)^5]$. However,
as easily seem from \cite[Claim 6.1]{LT_cmp}, the actual proof of
\cite[Th. 3.1]{LT_cmp} gives the stronger bound \eqref{eq:CMPbound}.
Note that, since by \cref{cor:L4} the mixing time satisfies $T_{mix}\ll \delta^{-5}$, by
the limit shape theorem \eqref{Eq:limshape} the event
$\max_{x\in W^\d}|h_t(x)-\phi(x)|\ge \epsilon$ has very small
probability also for $t\ge \delta^{-5}$.

Finally, a key ingredient of the present work is the following sharp bound on height fluctuations in domains with smooth limit shape:
\begin{theorem}\cite[Prop. 1.2-1.4]{L_arxiv}\label{prop:fluctuations_reference}  Let the domains $U,W$ and the limit shape $\phi:U\mapsto \R$ satisfy Assumptions \ref{ass:ls} and \ref{assu:W}. There exists a sequence of discrete domains  $W^\delta\subset \mathcal T_\delta$  and of boundary conditions $h_\partial:\partial W^\delta\to \delta\mathbb Z$  such that:
	\begin{itemize}
        \item $d_{\mathcal H}(W^\delta,W^\d)=O(\delta)$, where $W^\d$
          is defined in Section \ref{sec:somenot} and $d_{\mathcal H}$
          denotes Hausdorff distance;
		\item for all $v$ in $\partial W^\d$, $ | h_\partial(v) - \phi(v) | \leq C \delta $;
		% \item for all $v$ in $W^\d$, $|\pi_{U^\delta,h_\partial} ( h(v)) - \phi(v)| \leq C \delta |\log\delta|$;
		\item  for all $v$ in $W^\d$ and $n\geq 1$, $\pi_{W^\delta,h_\partial}\Big( \Big|h(v) -\phi(v)
                  % \pi_{W^\delta,h_\partial}(h(v))
                  \Big|^n \Big) \leq C_n \delta^n |\log\delta|^{2n}$.
                \end{itemize}
                
                The constants $C$ and $C_n$ for $n \geq 1$ depend on
                the domain $U$ and on the function $\phi:U\mapsto
                \R$. The dependence on $\phi$ is continuous for the
                topology induced by the Sobolev norm
                $\|\cdot\|_{k,p},k=33,p=3$ in $U$, i.e. the $L^p$ norm on all
                derivatives of $\phi$ up to order $k$ in $U$.
        \label{th:33}
\end{theorem}
We will actually use the following simple consequence, that follows immediately by monotonicity:
\begin{corollary}
  \label{cor:33}
  In Theorem \ref{th:33} one can take $W^\delta= W^\d$ and $h_\partial$ as the restriction to $\partial W^\d$ of any $K$-discretization of $\phi$. In this case, the constants $C_n$ depend also on $K$.
\end{corollary}

\subsection{Sketch of the proof of Theorem \ref{th:main}}\label{sec:overview}

Here, we sketch the strategy of the proof of the mixing time upper
bound. The key point is to show that, with probability at least
$1-\epsilon$, for any initial condition and any time
\[t\in I:= [\delta^{-2-\eta},\delta^{-5}],\] the height
function satisfies 
\begin{eqnarray}
  \label{eq:sati}
h^-\le h_t\le h^+  \text{ in the whole domain } W^\d,
\end{eqnarray}
for two well-chosen, time-independent functions $h^\pm:W^\d\mapsto \R$ such that
$\|h^--h^+\|_\infty \le \delta^{1-\eta}$ (This should be compared with \cref{thm:CMPbound}, where $\delta^{1-\eta}$ is replaced by the much bigger $\epsilon$). This is the outcome of
\cref{th:core}.   Therefore, in the time interval $I$ the true dynamics
and the one constrained to stay between $h^-$ and $h^+$ coincide with
probability $1-\epsilon$. On the other hand, the equilibrium measure
$\pi_{W^\d,\phi^\d|_{\partial W^\d}}$ and the constrained one (recall Remark \ref{rem:eqmes}) have total variation distance $o(1)$ as $\delta\to0$, as we will deduce from \cref{cor:33}.  Then, thanks to Lemma
\ref{lemma:PTRF} (which can be used to estimate the mixing time of the constrained dynamics), this implies the mixing time upper bound of Theorem
\ref{th:main} (with $\eta$ replaced by a constant times $\eta$, but $\eta$ is anyway arbitrary); see Section \ref{sec:UB} for
details on this implication. 

In order to show that \eqref{eq:sati} holds with probability
$1-\epsilon$, the idea is to introduce a sequence of deterministic
upper and lower bounds $h^\pm_i:W^\d\mapsto \R,i\ge0$ and to prove
that, for an appropriately chosen sequence of deterministic times
$t_i,i\ge0$, one has $h^{-}_{i} \leq h_t \leq h^{+}_{i}$ for all
$t \in [t_i, \delta^{-5}]$, with high probability.  The distance
$\|h_i^+-h^-_i\|_\infty$ will decrease with $i$, starting from a value $O(1)$
for $i=0$.  The core of the proof is to carry out the induction over
$i$ until a final step $i_{max}$ such that
$t_{i_{max}}=O(\delta^{-2+O(\eta)})$ and
$\|h^-_{i_{max}}-h^+_{i_{max}}\|_\infty\le \delta^{1-\eta}.$ See \cref{th:core}.

The general idea guiding the choice of $h^\pm_i$ and the times $t_i$
is as follows (by symmetry, we focus on the upper bounds on the
height). First of all, recall that the time $ \delta^{-5}$ is much
larger than the mixing time of the Glauber dynamics on $W^\d$ (by
\cref{cor:L4}) so we will ignore the restriction $t\le \delta^{-5}$ in this
sketch. Suppose by induction that, for some $i$, with high
probability, $h_t \leq h_{i}^+$ for all $t \geq t_i$. Fix some
$x \in W^\d$ and assume for simplicity that it is at a macroscopic (i.e. not vanishing as $\delta\to0$)
distance from the boundary $\partial W^\d$.
Let us
try to prove that with overwhelming probability
\begin{eqnarray}
  \label{eq:32}
h_t(x) \leq h^+_{{i+1}}(x):=h^+_{i}(x)-\delta   
\end{eqnarray}
for $t\ge t_{i+1}$, where $t_{i+1}$ will be determined in a
moment. (The actual relation between $h^+_i$ and $h^+_{i+1}$ will be
somewhat less trivial than $h^+_i-h^+_{i+1}=\delta$, see
\eqref{eq:hhpsi}, but for the purpose of this sketch we ignore this
issue).  By monotonicity and the induction hypothesis, it is enough to
restart the dynamics at time $t_i$ from (a discretization of)
$h^+_{i}$ and to let it evolve with the constraint that  $h_t\le h^+_{i}$. One of the
main ideas in the proof (which goes back to \cite{CMT}) is that
running the dynamics only inside a ball $B$ centered at $x$ and of any
radius $r_i$ gives an upper bound on the height function\footnote{The case of $x$ close to the boundary $\partial W^\d$ requires
special care at this step, because $B$ does not necessarily fit inside $W^\d$. This is where it is important that the macroscopic shape $\phi$ is defined in a
domain $U$ that contains $\overline W$, as in \cref{assu:W}.}. 
It turns out that the correct
choice for $r_i$ is such that $\|h^+_i-h^-_i\|_\infty \times r_i^2 \approx \delta^{1- \eta}$,
see more precisely \eqref{eq:ri}. In particular, one can verify that
$r_i\ll 1$ as $\delta\to0$ (see  \eqref{eq:r0rimax}). This choice guarantees
that the curvature of the height function plays a small but
non-negligible effect on the evolution inside $B$.

An argument involving monotonicity and Lemma \ref{lemma:PTRF} (see
Section \ref{sec:proof52}) shows that the dynamics restricted to the
ball $B$ reaches equilibrium in a time of order
$ r_i^2 \delta^{2-2\eta} \delta^{-4}% = \delta^{-2 -2
  % \eta + \alpha }
 $.  Taking $t_{i+1} - t_i\gtrsim r_i^2 \delta^{-2-2\eta}$ so that the
 dynamics is well mixed in $B$, we are reduced to the analysis of the
 equilibrium measure in the small region $B$, with boundary conditions
 $h^+_{i}|_{\partial B}$ on $\partial B$, which we view as a small
 perturbation of the limit shape $\phi$. More precisely, we need for
 \eqref{eq:32} to occur with high probability. For this, we will use
 \cref{cor:33} to control separately the expectation and the
 fluctuations of $h(x)$ with these boundary conditions.  Finally, we will easily verify (see
 Eq. \eqref{eq:tmax}) that the final time
 $t_{i_{max}}=O(\delta^{-2-O(\eta)})$ as desired.

 This overall idea essentially goes back to \cite{CMT} and was adapted
 in \cite{LT_cmp} to deal with non-affine limit shapes. However, in
 \cite{LT_cmp} we were only able to carry out the strategy up to a
 step $i_{max}$ such that
 $\| h^+_{i_{max}} - h^-_{i_{max}} \|\le \epsilon$, with $\epsilon$
 small but independent of $\delta$.  This allowed us to prove
 \cref{thm:CMPbound}, which is much weaker than the statement of
 Theorem \ref{th:main}. In the present work, \cref{thm:CMPbound} is
 used instead as the first step of the new iteration which works up to
 the step $i_{max}$ where
 $\| h^+_{i_{max}} - h^-_{i_{max}} \|\le \delta^{1-\eta}$, and allows
 us to obtain the essentially sharp mixing time upper bound in
 \eqref{eq:main}.  The crucial point limiting how far the iteration
 can be run is that we need to sharply control height fluctuations in
 balls $B$ of diameter $r_i$ for all $i\le i_{max}$, uniformly over
 the relevant boundary conditions on $\partial B$ produced by the
 dynamics. In \cite{LT_cmp}, we were limited to working in domains $B$
 where the curvature of the limit shape plays a non-negligible role,
 but its derivatives of order at least $3$ are negligible, so that a
 Taylor expansion of the limit shape up to order $2$ is a good
 approximation; this restricted us to the regime
 $r_{i_{max}} \sim \delta^{1/2-\eta}$. In contrast, here we will have
 $r_{i_{max}} \sim \delta^{\eta}\gg \delta^{1/2-\eta}$; in domains of
 such size, the limit shape solves a fully non-linear problem and this
 is where the recent equilibrium estimates from \cref{th:33} are
 crucial. It turns out that correctly applying \cref{th:33} will
 require some delicate information on perturbations of the non-linear
 Euler Lagrange PDE \eqref{eq:EL}, which will be stated in
 \cref{sec:statement_equilibrium} and proved in
 \cref{sec:perturbation_continuum}. The latter is the technical core
 of the proof.

\section{Mixing time bounds}
\label{sec:UB}

In this section, we provide a precise version of the inductive statement
mentioned above and, assuming its validity, we conclude the proof of
the mixing time upper bound.
(Let us recall that the lower bound instead is proven exactly like in the case
where the limit shape $\phi$ is affine,
see
\cite[Sec. 5.2]{LT_ptrf}.)

Our goal is to prove, say,
\begin{eqnarray}
  \label{Tmixub}
  T_{mix}=O(\delta^{-2-10\eta})
\end{eqnarray}
for every $\eta>0$. 
Define the function $\psi:\mathbb R^2\mapsto \mathbb R$ as 
\begin{eqnarray}
  \label{eq:psi}
\psi(x)=  \psi(x_1,x_2)=c-e^{ x_1/\xi}-e^{x_2/\xi}
\end{eqnarray}
where $\xi>0$ will be chosen small later and $c$ is such that
$\min_{x=(x_1,x_2)\in U}\psi(x_1,x_2)=1$. Note that (with $\nabla^2\psi$ the Hessian matrix of $\psi$)
\begin{eqnarray}
  \label{eq:gradpsi}
  \nabla\psi=-\frac1\xi(e^{x_1/\xi},e^{x_2/\xi}), \quad \nabla^2\psi=-\frac1{\xi^2}
\left(  \begin{array}{ll}
          e^{x_1/\xi}& 0\\
          0 & e^{x_2/\xi}
        \end{array}
  \right)=:-\left(  \begin{array}{ll}
          m^2_1(x)& 0\\
          0 & m_2^2(x)
        \end{array}
  \right).
\end{eqnarray}
In particular, the second derivatives are negative and
\begin{eqnarray}
  \label{eq:confrontoderivate}
  \max(-\partial^2_{x_1} \psi,-\partial^2_{x_2}\psi)\ge \frac1\xi \max(|\partial_{x_1}\psi|,|\partial_{x_2}\psi|).
\end{eqnarray}
This specific choice of $\psi$ will be crucial in the proof of \cref{thm:equilibrium_development} (see \eqref{eq:ato0}), but until then the reader can think of $\psi$ as any nice strictly concave function.
Given $w\in W$, we define the positive quadratic form $Q_w$ as
\begin{eqnarray}
  \label{eq:quadform}
x\mapsto Q_w(x):=-\langle(x-w),\nabla^2\psi(w) , (x-w)\rangle\ge 0.
\end{eqnarray}

We let  $t_0=\delta^{-2-\eta}$ and we fix a constant $\epsilon_0>0$ that will be chosen small enough later.  We define a set  $A_0$ of height functions  as
\begin{eqnarray}
  \label{eq:A}
  A_0=\Big\{h\in \Omega_{W^\d,\phi^\d|_{\partial W^\d}}:\mathbb P_h\Big(|h_t-\phi|<\epsilon_0\psi\text{ for all }t\in [0,\delta^{-5}-t_0]\Big)\geq 1 -\sqrt{\epsilon_0}\Big\}
\end{eqnarray}
where $\P_h$ denotes the law of the process started from $h$ and
the inequality $|h_t-\phi|\le \epsilon_0\psi$ is intended as holding pointwise
in $W^\d$.
Note that  Theorem \ref{thm:CMPbound} plus the Markov property of the dynamics
says that, for any initial condition $h$,
\begin{eqnarray}
  \mathbb E_h\Big[\P_{h_{t_0}}\Big(\exists t\in [0,\delta^{-5}-t_0],x\in W^\d:|h_t(x)-\phi(x)|\geq\epsilon_0\psi(x)\Big)\Big]\le 
  \epsilon_0.
  \end{eqnarray}
  Therefore, by Markov's inequality,
  \begin{multline}
    \P_h(h_{t_0}\not\in A_0)\\
    =\P_h\Big[\P_{h_{t_0}}\Big(\exists t\in [0,\delta^{-5}-t_0],x\in W^\d:|h_t(x)-\phi(x)|\ge \epsilon_0 \psi(x)\Big)>\sqrt{\epsilon_0}\Big]\le \sqrt{\epsilon_0},
  \end{multline}
  i.e, for any initial condition $h$
  \begin{eqnarray}
    \label{eq:ie}
    \P_h(h_{t_0}\in A_0)\ge 1-\sqrt{\epsilon_0}.
  \end{eqnarray}
For $0\le i\le
i_{max}:=(\epsilon_0-\delta^{1-4\eta})/\delta$,\footnote{It is
  understood that $i_{max}$ is the integer part of this value, but we
  drop all ``integer parts'' here and in the following, for lightness
  of notation. } define $\eps_i$, $r_i$ and $t_i$ by induction by
	\[
\label{eq:ri}	\epsilon_i:=\epsilon_0-\delta i, \quad r_i := \sqrt{\delta^{1-\eta}/\epsilon_i}, \quad t_{i+1}-t_i=\left(\frac{r_i}\delta\right)^{2}\delta^{-5\eta}= \frac{\delta^{-1-6\eta}}{\epsilon_i}.
	\]
	Note that $\eps_i$ is increasing, $r_i$ is increasing,
	\[\label{eq:r0rimax}
	r_0 = \sqrt{\delta^{1-\eta}/\epsilon_0}, \quad \epsilon_{i_{max}}=\delta^{1-4\eta}, \quad r_{i_{max}}= \delta^{3\eta/2}\ll 1,
	\]
	and
	\[\label{eq:tmax}
	t_{i_{max}}=t_0+\delta^{-2-6\eta}\sum_{k=\delta^{-4\eta}}^{\epsilon_0/\delta}\frac1k=
	O\Big(\delta^{-2-6\eta}\log(\epsilon_0/\delta)\Big)=O(\delta^{-2-7\eta}).
	\]

  Similarly to Eq. \eqref{eq:A}, we also define the following
  $i$-dependent sets of height functions $A_i, i\ge1$:
\begin{eqnarray}
\label{eq:Ai}
A_i=\Big\{h\in \Omega_{W^\d,\phi^\d|_{\partial W^\d}}:\mathbb P_h\Big(|h_t-\phi|<\epsilon_i\psi \text{ for all }t\in [0,\delta^{-5}-t_i]\Big)\geq 1 -2i\delta^3 \Big\}.
\end{eqnarray}
Note that this definition effectively corresponds to choosing the
functions $h^+_i$ mentioned in Section \ref{sec:overview} as
\begin{eqnarray}
  \label{eq:hhpsi}
  h_i^\pm=\phi\pm\epsilon_i \psi.
\end{eqnarray}

The inductive statement underlying the core of the proof is the following:
\begin{theorem}\label{th:core}
There exists $\delta_0>0 $ such that for every $0 \leq i <
i_{max}$, for every $h \in A_i$ and if the lattice mesh satisfies $\delta<\delta_0$, on has
        \begin{eqnarray}
          \label{eq:core}
	\P_h( h_{t_{i+1} - t_{i}} \in A_{i+1}) \geq 1 - \delta^2.   
        \end{eqnarray}
\end{theorem}

\begin{proof}[Proof of \eqref{eq:main} (upper bound) given Proposition
  \ref{th:core}]
  Like in the argument leading from \cref{thm:CMPbound} to
  \eqref{eq:ie}, we apply successively at all times $t_i$,
  $ 0 \leq i < i_{max}$ the Markov property and Markov's
  inequality. Thanks to a union bound on
  $i\le i_{max}\le \delta^{-1}$, we see that, for any initial
  condition $h$, 
  \[
  \P_h( h_{t_{i_{max}}} \in A_{i_{max}}) \geq 1 - \sqrt{\epsilon_0} - \delta^2 i_{max}\geq 1 - 2 \sqrt{\epsilon_0}
\]
if $\delta$ is small enough.
Let $T=\delta^{-2-10\eta}$. We have then, with $\mu_T^h$ denoting the law of the process $h_T$ at time $T$ with initial condition $h$ and $\pi:=\pi_{W^\d,\phi^\d|_{\partial W^\d}}$,
\begin{multline}
 \|\mu_T^h-\pi\|= \max_B\Big(\P_h(h_T\in B)-\pi(B)\Big)\\\le \max_B\Big(\P_h(h_T\in B, h_{t_{i_{max}}} \in A_{i_{max}})-\pi(B)\Big)+2 \sqrt{\epsilon_0}.
\end{multline}
Next, note that by the Markov property
\begin{eqnarray}
  \P_h(h_T\in B, h_{t_{i_{max}}} \in A_{i_{max}})\le \max_{h'\in A_{i_{max}}}\P_{h'}(h_{T-t_{i_{max}}}\in B).
\end{eqnarray}
Let $\tau$ be the stopping time
\begin{eqnarray}
  \label{eq:tau}
  \tau:=\inf\{t\ge0:|h_t(x)-\phi(x)|\geq \epsilon_{i_{max}}\psi(x) \text{ for some } x\in W^\d\}.
\end{eqnarray}

For $h'\in A_{i_{max}}$, the probability that 
$\tau< \delta^{-5}-t_{i_{max}}$ is at most
$2i_{max}\delta^3\leq \delta^{3/2}$. On the event $\tau\ge  \delta^{-5}-t_{i_{max}}$, instead, the evolution in the time interval $[0,\delta^{-5}-t_{i_{max}}]$ can be perfectly coupled with the dynamics constrained
between $\phi-\delta^{1-4\eta}\psi$ and
$\phi+\delta^{1-4\eta}\psi$, whose law we denote here  by $\hat\P_{h'}$ (here we have used that $\epsilon_{i_{max}}=\delta^{1-4\eta}$). Noting that $T-t_{i_{max}}\le \delta^{-5}-t_{i_{max}}$, this implies
\begin{multline}
  \label{dpsc}
  \P_h(h_T\in B, h_{t_{i_{max}}} \in A_{i_{max}}) \\\le \max_{h'\in A_{i_{max}}}\Big[\P_{h'}(h_{T-t_{i_{max}}}\in B;\tau< \delta^{-5}-t_{i_{max}})+\P_{h'}(h_{T-t_{i_{max}}}\in B;\tau\geq \delta^{-5}-t_{i_{max}})\Big]\\
  \le \delta^{3/2}+ \max_{h'\in A_{i_{max}}}\hat\P_{h'}(h_{T-t_{i_{max}}}\in B).
\end{multline}
By \cref{lemma:PTRF}, the constrained dynamics has mixing time upper 
bounded by $\delta^{-2-9\eta}$, for $\delta$ small. We note also that
\[\delta^{-2-9\eta}\ll (1/2)\delta^{-2-10\eta}\le T-t_{i_{max}}(\ll \delta^{-5}-t_{i_{max}})\]
as $\delta\to0$. Therefore, using the sub-multiplicative property
\eqref{eq:submult}, the r.h.s. of \eqref{dpsc} is upper bounded by
\begin{eqnarray}
  \hat \pi(B)+\delta^{3/2}+\epsilon_0
\end{eqnarray}
for $\delta$ small enough, uniformly in $h'$, with $\hat\pi$ the equilibrium measure of the constrained dynamics. To 
conclude, we recall (see \cref{rem:eqmes}) that $\hat\pi$ is simply $\pi_{W^\d,\phi^\d|_{\partial W^\d}}$ conditioned  to the event $\{\phi-\delta^{1-4\eta}\psi\le h\le \phi+\delta^{1-4\eta}\psi\}$ and that, by
  \cref{cor:33}, this event occurs under $\pi_{W^\d,\phi^\d|_{\partial W^\d}}$ with probability $1+o(1)$ as $\delta\to0$, and in particular with probability at least $1-\epsilon_0$. Wrapping up, we have obtained that
    \begin{eqnarray}
\|\mu_T^h-\pi\|\le \delta^{3/2}+2\epsilon_0 + 2 \sqrt{\epsilon_0} \le \frac14
    \end{eqnarray}
    if $\delta$ and $ \epsilon_0$ are small enough, which implies \eqref{Tmixub}, by definition of mixing time and the choice $T=\delta^{-2-10\eta}$.
\end{proof}

\section{Proof of Theorem \ref{th:core}}
\label{sec:core}

In this section, we consider a fixed $i\le i_{max}$ and we reduce the inductive
statement of \cref{th:core} to two sharp statements on height
fluctuations for some equilibrium measures with suitable non-affine boundary
height, as mentioned in \cref{sec:overview}.

Note that by symmetry it is enough to prove the analogous statement
\eqref{eq:core} for the event
\[
A'_{i+1}=\Big\{h:\mathbb P_h\Big(h_t-\phi<\eps_{i+1}\psi \text{ for all }t\in [0,\delta^{-5}-t_{i+1}]\Big)\geq 1 -2(i+1)\delta^3 \Big\}
\]
with $\delta^2$ replaced by $\delta^2/2$ in the right-hand side of \eqref{eq:core},
since the probability of deviations of $h_t$ below $\phi-\epsilon_{i+1}\psi$ can be estimated analogously.

In analogy with \eqref{eq:tau}, define the stopping time
\begin{eqnarray}
  \label{eq:taui}
    \tau_i:=\inf\{t\ge0:|h_t(x)-\phi(x)|\geq \epsilon_{i}\psi(x) \text{ for some } x\in W^\d\}
\end{eqnarray}
and note that on the event
$\{\tau_i\ge \delta^{-5}-t_i(\ge t_{i+1}-t_i)\}$, we can perfectly couple
the dynamics on $[0,\delta^{-5}-t_i] $ with the one constrained
between $\phi-\epsilon_i\psi$ and $\phi+\epsilon_i\psi$, whose law we
denote $\hat\P_h$.  For $h\in A_i$ we have
\begin{multline}
  \P_h(h_{t_{i+1}-t_i}\in A'_{i+1})\ge  \P_h(h_{t_{i+1}-t_i}\in A'_{i+1};\tau_i\ge\delta^{-5}-t_i)\\
  \ge\hat \P_h(h_{t_{i+1}-t_i}\in A'_{i+1}) - \P_h(\tau_i < \delta^{-5}-t_i)\\
  \ge \hat \P_h(h_{t_{i+1}-t_i}\in A'_{i+1})-2i\delta^3
\end{multline}
because, by definition of $A_i$, $\P_h(\tau_i<\delta^{-5}-t_i)\le 2i\delta^3$.
Since $2i\delta^3\le 2i_{max}\delta^3\le 2\epsilon_0\delta^2\le \delta^2/4$ for $\epsilon_0$ small, it is enough to prove that, for the
constrained dynamics, $h_{t_{i+1} - t_i} \in A'_{i+1}$ with
$\hat\P_h$-probability at least $1-\delta^2/4$.  By monotonicity, the worst case for the
constrained dynamics is to start from the highest configuration
$h^{(i)}_{max}\in \Omega_{W^\d,\phi^\d|_{\partial W^\d}}$ lower than $\phi + \eps_i \psi$, i.e., we need
\begin{eqnarray}
  \label{eq:worst}
  \hat \P_{h^{(i)}_{max}}(h_{t_{i+1}-t_i}\in A'_{i+1})\ge 1-\frac{\delta^2}4.
\end{eqnarray}
Assume that we can prove
\[\label{eq:induction}
\hat\P_{h^{(i)}_{max}} \big( h_t < \phi + \eps_{i+1} \psi  \text{ for every } t\in [t_{i+1} - t_i, \delta^{-5}- t_i]\big) \geq 1 - \delta^6.
\]
Applying the Markov property as for $A_0$ above, this is equivalent to
\begin{multline}
  \hat \E_{h^{(i)}_{max}}\hat \P_{h_{t_{i+1}-t_i}}\Big( h_t(x)\ge \phi(x)+\epsilon_{i+1}\psi(x) \text{ for some } x\in W^\d, t\le \delta^{-5}-{t_{i+1}}\Big)\leq \delta^6.
\end{multline}
Via Markov's inequality, this gives
\begin{eqnarray}
  \hat \P_{h^{(i)}_{max}}\Big[\hat\P_{h_{t_{i+1}-t_i}}\big(h_t<\phi+\epsilon_{i+1}\psi \text{ for all }
  t\le \delta^{-5}-{t_{i+1}}\big)\ge 1-\delta^3\Big]\ge 1-\delta^3,
\end{eqnarray}
that is, since $\delta^3< 2(i+1)\delta^3$,
\begin{eqnarray}
  \hat \P_{h^{(i)}_{max}}(h_{t_{i+1}-t_i}\in A'_{i+1})\ge 1-\delta^3\ge1-\frac{\delta^2}4
\end{eqnarray}
and \eqref{eq:worst} follows for $\delta$ small.
Therefore, we have reduced our task to proving \eqref{eq:induction}.

\subsection{Some sharp  equilibrium results on mesoscopic scales}\label{sec:statement_equilibrium}

The crucial ingredients for the proof of Eq. \eqref{eq:induction} are two sharp equilibrium fluctuation statements on mesoscopic scales.
\begin{definition}
  \label{def:Eiw}
Given $w \in W^\d$, 
let $E_{i,w}$ be the ellipse centered at $w$ and determined by the equation
\begin{eqnarray}
  \label{eq:qformconst}
x\in E_{i,w}\Leftrightarrow  Q_w(x)\le  r_i^2
\end{eqnarray}
with $Q_w$ the quadratic form  in \eqref{eq:quadform}.
Note that the horizontal (resp. vertical) axes of the ellipse are of length
$r_i/m_1(w)$ (resp. $r_i/m_2(w)$).
\end{definition}
\begin{remark}
  \label{rem:ellissi}
The family of ellipses $\{E_{i,w}\}_{i\le i_{max}, w\in W}$ has aspect ratio uniformly
bounded away from zero and infinity (once $\xi$ and $W$ are given). Note also that the ellipses have axes parallel to the Cartesian axes, because the matrix $\nabla^2\psi$ is diagonal.  
\end{remark}
 Note that by
our choice of $r_i$ we have that, for $x\in E_{i,w}$,
\begin{eqnarray}
  \label{eq:skoda}
  \phi(x)+  \epsilon_{i}\psi(x)
  =\phi(x)+\epsilon_i\psi(w)+\epsilon_i \langle \nabla\psi(w), x-w\rangle-\frac{\epsilon_i}2
 Q_w(x)+o(\delta)
\end{eqnarray}
where $o(\delta)$ is in fact $O(r_i^3\epsilon_i)=O(
\delta^{1+\eta/2})$ uniformly in $i,w$ (the constants implicit in the
error terms can depend on the parameter $\xi$ that enters the
definition of $\psi$).  By construction of the ellipse,
\begin{eqnarray}
  \label{bcku}
  \epsilon_{i}\psi(x)=\epsilon_i\Big(C_{i,w}+ \langle \nabla\psi(w), x-w\rangle\Big)
  +o(\delta) \;\text{for}\; x\in \partial E_{i,w}
\end{eqnarray}
where
\begin{eqnarray}
  \label{eq:C'}
  C_{i,w}=\psi(w)-\frac12%% \mathcal C_w 
  r_i^2.  
\end{eqnarray}
 We emphasize that even for $w$ close to $\partial W^\d$, both $\phi$ and $\psi$ are well defined over the whole ellipse $E_{i,w}$ because  $r_i\ll1$ for all $i\le i_{max}$ and therefore $E_{i,w}$ fits in $U$.

 \begin{definition}
   \label{def:phiiw}
 We let $\phi_{i,w}:E_{i,w}\mapsto \R$ denote the limit shape in $E_{i,w}$, that is  the solution of the PDE \eqref{eq:EL} in
 $E_{i,w}$ with boundary condition on $\partial E_{i,w}$ given by
\begin{eqnarray}
  \label{eq:bcku1}
f_{i,w}(x):= \phi(x)+\epsilon_i\Big(C_{i,w}+ \langle \nabla\psi(w), x-w\rangle\Big),\;x\in \partial E_{i,w}.
\end{eqnarray}
 \end{definition}
   Comparing with \eqref{bcku}, we see that the b.c. is $o(\delta)$ away
from $[\phi+\epsilon_i\psi]|_{\partial E_{i,w}}$.  Note also that the boundary condition of
$\phi_{i,w}$ is just $\phi|_{\partial E_{i,w}}$, up to an additive
constant and a linear function. The additive constant has a trivial
effect on the limit shape $\phi_{i,w}$; the linear does not, because
the PDE \eqref{eq:EL}  is non-linear.

Theorems \ref{thm:equilibrium_development} and
\ref{th:letrucdestrentederivees} below summarize the information we
need on the local equilibrium in $E_{i,w}$, respectively on the limit
shape $\phi_{i,w}$ itself and on fluctuations around it.

\begin{theorem}\label{thm:equilibrium_development}
  The limit shape $\phi_{i,w}$ in $E_{i,w}$ satisfies
  \begin{eqnarray}
    \label{eq:sa3}
    \phi_{i,w}(x)=\epsilon_i C'_{i,w}+\phi(x)+\epsilon_i \langle \nabla\psi(w), x-w\rangle-a\epsilon_i Q_w(x) +o(\delta)+O(\epsilon_i^2 r_i^2)
  \end{eqnarray}
  where $Q_w$ is the positive quadratic form \eqref{eq:quadform} and
  \begin{eqnarray}
    \label{eq:C''}
    C'_{i,w}=\psi(w)-\Big(\frac12-a\Big)%% \mathcal C_{w}
    r_i^2
  \end{eqnarray}
 is independent of $x$, the constant $a$ is smaller than $1/4$ for $\xi$ small (uniformly in 
  $\epsilon_0,w,i\le i_{max},\delta$) and the  constants implicit in the error terms are  uniform w.r.t.
  $w,i\le i_{max}, x\in E_{i,w}$.
\end{theorem}

\begin{remark}
  For any ellipse $E^\rho_{i,w}$ with the same center and aspect ratio as $E_{i,w}$, just
  shrinked by a factor $\rho<1$, \cref{thm:equilibrium_development} implies that
  \begin{eqnarray}
    \label{eq:psinonpsi}
    \big[ \phi_{i,w}\big]|_{\partial E^\rho_{i,w}}- \big[ \phi+\epsilon_i \psi\big]|_{\partial E^\rho_{i,w}}=-\epsilon_i r_i^2 %% \mathcal C_w
    \Big(\frac12-a\Big)(1-\rho^2)+o(\delta)+O(\epsilon_i^2r_i^2)
  \end{eqnarray}
  and the r.h.s. is strictly negative for $\rho<1$: the limit shape $\phi_{i,w}$ is strictly lower than
  $\phi+\epsilon_i\psi$, in the interior of $E_{i,w}$.
\end{remark}

\begin{definition}
 We let $\pi_{i,w}$ be the uniform
distribution on height functions on the discretized domain $(E^{1/2}_{i,w})^\d$ with
boundary condition given by (any 1-discretization of) $\phi_{i,w}$ on
$\partial (E^{1/2}_{i,w})^\d$.
\label{piiw}
\end{definition}

\begin{theorem}
  \label{th:letrucdestrentederivees}
  For every $n\ge1$ there exists a constant $C_n > 0$ such that for all $w\in W^\d,i\le i_{max}$,
  all $x$ in $(E^{1/2}_{i,w})^\d$,
    \begin{eqnarray}
    \label{eq:letruc}
    \pi_{i,w}\left[ \Big| h(x) - \phi_{i,w}(x) \Big|^n\right] \leq C_n \delta^{n} \abs{\log\delta}^{2n}. 
  \end{eqnarray}
\end{theorem} 
Let us emphasize here that compared to \cref{prop:fluctuations_reference}, the main difference in the above statement is the uniformity of the constants $C_n$ with respect to $w$ and $i$.

Theorems \ref{thm:equilibrium_development} and \ref{th:letrucdestrentederivees} will be proven in \cref{sec:perturbation_continuum}. For the moment, let us assume they hold and let us complete the proof of \eqref{eq:induction}.

\subsection{Proof of Equation \eqref{eq:induction}}
\label{sec:proof52}

We will prove \eqref{eq:induction} by proving that, for $\delta<\delta_0$ for some small but positive $\delta_0$,
\[\label{eq:induction2}
\hat\P_{h^{(i)}_{max}} \Big( h_t(w) < \phi(w) + \eps_{i+1} \psi(w)  \text{ for every } t\in [t_{i+1} - t_i, \delta^{-5}- t_i]\Big) \geq 1 - \delta^9
\]
separately for each $w\in W^\d$ and then applying a union bound, since
the cardinality of $W^\d$ is $O(\delta^{-2})$.  For the rest of this
section, $\delta_0$ can change from line to line but it is always a
small constant that is independent of $i,w$.

Recall that $\hat\P_{h^{(i)}_{max}}$ denotes the law of the process in
the whole domain $W^\d$, with height constrained between
$\phi-\epsilon_i\psi$ and $\phi+\epsilon_i\psi$, and with initial
condition $h^{(i)}_{max}$ that is the maximal height function in
$\Omega_{W^\d,\phi^\d|_{\partial W^\d}}$ that is lower than
$\phi+\epsilon_i\psi$.

First we observe that, since the dynamics starts from the maximal
condition allowed by the constraints, by monotonicity one can censor
any update outside of $W^\d \cap (E^{1/2}_{i,w})^\d$. In order to
treat on the same footing the points $w$ close and far  from the
boundary of $W^\d$, it is convenient to have the height function evolving on the
whole discretized ellipse $(E^{1/2}_{i,w})^\d$, even when this does not
entirely fit in $W^\d$. This is done via a chain of monotonicity
arguments in \cref{prop:quelbordel} below. For this, we need a couple
of auxiliary definitions:
\begin{definition}
  Let $h_{max}^{(i,w)}:(E^{1/2}_{i,w})^\d\mapsto \delta\Z$ denote the
  maximal height function on $(E^{1/2}_{i,w})^\d$  lower than
  $\phi+\epsilon_i\psi$. Let $\tilde\P_{i,w}$ be the law of the
  Glauber dynamics in $(E^{1/2}_{i,w})^\d$ started from the initial
  condition $ h_{max}^{(i,w)}$, evolving with height function constrained
  between $\phi_{i,w}-\delta^{1-2\eta}$ and $\phi+\epsilon_i\psi$, and
  with boundary height given by the restriction
  $h_{max}^{(i,w)}|_{\partial (E^{1/2}_{i,w})^\d}$.
\end{definition}

\begin{proposition}
  \label{prop:quelbordel}
  The probability in the l.h.s. of \eqref{eq:induction2} is lower bounded by
  \begin{eqnarray}
    \label{eq:tildep}
    \tilde \P_{i,w}\big( h_t(w) < \phi(w) + \eps_{i+1} \psi(w)  \text{ for every } t\in [t_{i+1} - t_i, \delta^{-5}- t_i]\big).
  \end{eqnarray}
\end{proposition}
We emphasize that the main difference between the processes of law
$\hat \P_{h^{(i)}_{max}}$ and $\tilde \P_{i,w}$ is that in the former, the height evolves in
$W^\d$ while in the latter it evolves in the discrete ellipse
$(E^{1/2}_{i,w})^\d$. For most points $w$ one has
$W^\d\supset (E^{1/2}_{i,w})^\d$ but, for $w$ close to the boundary,
this is not true and the implication of \cref{prop:quelbordel} requires extra care.

\begin{proof}[Proof of \cref{prop:quelbordel}]

Let us note first of all that
\begin{eqnarray}
  \label{eq:solsol}
  \phi< \phi_{i,w}-\delta^{1-2\eta} \text{ on } E^{1/2}_{i,w}.
\end{eqnarray}
To see this, start by observing that on $\partial E_{i,w}$ one has
\begin{eqnarray}
  \label{eq:psipsipsi}
  \phi_{i,w}-\delta^{1-2\eta}=\phi+\epsilon_i \psi-\delta^{1-2\eta}\ge \phi+\delta^{1-4\eta}-\delta^{1-2\eta}>\phi
\end{eqnarray}
because $\psi\ge 1$ (as observed just after \eqref{eq:psi}) and
$\epsilon_i\ge \epsilon_{i_{max}}=\delta^{1-4\eta}.$ Since both
$\phi_{i,w}$ and $\phi$ are solutions of \eqref{eq:EL} in $ E_{i,w}$
and the former has higher boundary condition than the latter, by the maximum principle for elliptic PDE the
inequality is preserved in the interior of the domain, and in
particular on $ E^{1/2}_{i,w}$.
  
If $(E^{1/2}_{i,w})^\d\subset W^\d$, the statement of  \cref{prop:quelbordel}
is an immediate consequence of the following three observations:
  \begin{enumerate}
  \item [(i)] one can lower bound the probability in \eqref{eq:induction2}
    by censoring the updates outside $(E^{1/2}_{i,w})^\d$ (because the
    initial condition is maximal). In order to avoid introducing a new
    notation, we still call $\hat \P_{h^{(i)}_{max}}$ the law of the
    censored process. Now, both under $\hat \P_{h^{(i)}_{max}}$ and $\tilde\P_{i,w}$, the
    height function evolves only inside the domain $(E^{1/2}_{i,w})^\d$;
    
  \item [(ii)] the lower constraint $h_t\ge \phi_{i,w}-\delta^{1-2\eta}$ of
    the dynamics with law $\tilde \P_{i,w}$ is more stringent than the
    constraint $h_t\ge \phi-\epsilon_i\psi$ of the dynamics $\hat \P_{h^{(i)}_{max}}$,
    in view of \eqref{eq:solsol} and of $\psi\ge 0$;
    
  \item [(iii)] the two dynamics have the same upper constraint $h_t\le \phi+\epsilon_i\psi$;
  \item [(iv)] the initial condition $h^{(i)}_{max}|_{(E^{1/2}_{i,w})^\d}$ of
    the dynamics $\hat \P_{h^{(i)}_{max}}$ is lower (or equal) to that of the dynamics
    $\tilde \P_{i,w}$, because for the latter one takes the highest
    configuration lower than $\phi+\epsilon_i\psi$, while for the
    former one additionally requires that
    $h^{(i)}_{max}\in \Omega_{W^\d,\phi^\d|_{\partial W^\d}}$. The
    same observation holds for the boundary conditions on
    $\partial (E^{1/2}_{i,w})^\d$: the one of the dynamics $\hat \P_{h^{(i)}_{max}}$
    is lower or equal to that of the dynamics $\tilde\P_{i,w}$.
    
  \end{enumerate}

The argument for $w$ sufficiently close to the boundary of $W^\d$, so
that $(E^{1/2}_{i,w})^\d$ does not fit in it, is slightly more
involved.  In this case, under $\hat\P_{h^{(i)}_{max}}$ the height in
$(E^{1/2}_{i,w})^\d\setminus W^\d$ is time-independent. In fact, the height in
$(E^{1/2}_{i,w})^\d\setminus W^\d$ can be imagined to be fixed to any
configuration compatible with the actual boundary condition $\phi^\d$
on $\partial W^\d$; in particular, it is convenient to imagine that
$h$ is fixed to $\phi^\d$ on the whole $(E^{1/2}_{i,w})^\d\setminus
W^\d$. In this case, one sees easily that the steps (i) to (iv) above
again imply the statement of the proposition, once one adds the extra
observation that the height in $(E^{1/2}_{i,w})^\d\setminus W^\d$ of the process with law $\tilde \P_{i,w}$ is deterministically higher than the
(time-independent) one of the process with law $\hat \P_{h^{(i)}_{max}}$, because by
\eqref{eq:psipsipsi} $\phi^\d$ is lower than
$\phi_{i,w}-\delta^{1-2\eta}$.
\end{proof}

We are then left with the task of showing:
\begin{eqnarray}
  \label{eq:leftto}
\text{for $\delta<\delta_0$, the l.h.s. of 
  \eqref{eq:tildep} is  lower bounded by $1-\delta^9$.}
\end{eqnarray}

A first observation in this respect is that the mixing time of the
dynamics with law $\tilde\P_{i,w}$ is upper bounded
by \[c_0\delta^{-4}(\log \delta)^2r_i^2\Big(\epsilon_i
r_i^2+\delta^{1-2\eta}\Big)^2\] for some absolute constant $c_0$ that can
change from line to line in the rest of the proof.   To see this, it is sufficient to recall \cref{lemma:PTRF}, together with the
fact that the dynamics is constrained between
$\phi_{i,w}-\delta^{1-2\eta}$ and $ \phi+\epsilon_i\psi$ and to
observe that on
$E^{1/2}_{i,w}$, \[|(\phi_{i,w}-\delta^{1-2\eta})-(\phi+\epsilon_i\psi)|\le
c_0(\delta^{1-2\eta}+\epsilon_i r_i^2),\] as follows from
\eqref{eq:psinonpsi}%%  since $\mathcal C_w$ is uniformly
%% bounded
. We emphasize that
the  constant $c_0$ is uniform with respect to $i\le i_{max},w\in W^\d$.
Recalling that by definition $\epsilon_i
r_i^2=\delta^{1-\eta}$ (see \eqref{eq:ri}), we see that the mixing time is upper bounded
by
\[
c_0\Big(\frac{r_i}\delta\Big)^2\delta^{-4\eta}=c_0\delta^{\eta}(t_{i+1}-t_i).
\]
Therefore, by the sub-multiplicative property \eqref{eq:submult}, at times $t\ge t_{i+1}-t_i$ the chain is at total variation distance  $2^{-1/(c_0\delta^{\eta})}$ from its equilibrium measure $\tilde \pi_{i,w}$.
We have then that the probability in \eqref{eq:tildep} is lower bounded by
\begin{eqnarray}
  \tilde \P^{eq}_{i,w}\big( h_t(w) < \phi(w) + \eps_{i+1} \psi(w)  \text{ for every } t\in [0, \delta^{-5}]\big)-2^{-1/(c_0\delta^{\eta})},
\end{eqnarray}
where $\tilde \P^{eq}_{i,w}$ denotes the law of the equilibrium process.
We claim that under $\tilde\pi_{i,w}$ one has
\begin{eqnarray}
  \label{eq:100}
\tilde\pi_{i,w}( h(w) < \phi(w) + \eps_{i+1} \psi(w) )\ge 1-\delta^{18}
\end{eqnarray}
for $\delta\le \delta_0$, where the exponent $18$ has been chosen simply so that we will get
$\delta^9$ in \eqref{eq:leftto}, but it could be replaced by any other
positive number, at the price of changing $\delta_0$.

Assume for a moment that \eqref{eq:100} holds; since the cardinality of $(E^{1/2}_{i,w})^\d$ is smaller than $\delta^{-2}$ we deduce from \cref{cor:ubt} that
\begin{eqnarray}
 \tilde \P^{eq}_{i,w}\big( h_t(w) < \phi(w) + \eps_{i+1} \psi(w)  \text{ for every } t\in [0, \delta^{-5}]\big)\ge 1-\delta^{10}.
\end{eqnarray}
Putting everything together, \eqref{eq:leftto} follows, provided we prove \eqref{eq:100}.

Recall from \cref{rem:eqmes} that
\begin{eqnarray}
  \label{eq:condizionato}
\tilde\pi_{i,w}=\pi_{i,w}(\cdot|\phi_{i,w}-\delta^{1-2\eta}\le h\le \phi+\epsilon_i\psi),
\end{eqnarray}
where $\pi_{i,w}$ is given in \cref{piiw}.  Since the event in
\eqref{eq:100} is decreasing, in \eqref{eq:condizionato} we can drop
the conditioning on $h\le \phi+\epsilon_i\psi$. The conditioning on
$h\ge \phi_{i,w}-\delta^{1-2\eta}$, instead, cannot be dropped by
monotonicity. However, here \cref{th:letrucdestrentederivees} enters
into play.  Choosing $n=n(\eta)$ large enough in \eqref{eq:letruc} and
applying Tchebyshev's inequality, we see that
\[\pi_{i,w}(h\ge \phi_{i,w}-\delta^{1-2\eta})\ge 1-\delta^{18},\] so that 
the entire conditioning can be dropped in the definition of $\tilde\pi_{i,w}$ and it is enough
to prove \eqref{eq:100} for the unconditional  measure $\pi_{i,w}$.

To this end, note first of all that  \eqref{eq:psinonpsi} taken at $\rho=0$ implies that
\begin{eqnarray}
  \phi_{i,w}(w)= \phi(w)+\epsilon_i\psi(w)-\epsilon_ir_i^2%% \mathcal C_w
  (1/2-a)+o(\delta)+O(\epsilon_i^2r_i^2). 
\end{eqnarray}
Recall that %% $\mathcal C_w>0$ is uniformly bounded away from $0$, that
$\epsilon_ir_i^2=\delta^{1-\eta}$ and that $a$ can be assumed to be smaller than $1/4$, see \cref{thm:equilibrium_development}. We have then
\begin{eqnarray}
  \phi_{i,w}(w)< \phi(w)+\epsilon_i\psi(w)-\delta^{1-\eta} c_0
\end{eqnarray}
for some constant $c_0>0$ uniform with respect to $w,i$. Finally, recalling that $\epsilon_i-\epsilon_{i+1}=\delta$ and that $\psi$ is bounded, we see that also
\begin{eqnarray}
  \phi_{i,w}(w)< \phi(w)+\epsilon_{i+1}\psi(w)-\delta^{1-\eta} \frac{c_0}2.
\end{eqnarray}
Therefore, we have
\begin{eqnarray}
  \pi_{i,w}\Big( h(w) < \phi(w) + \eps_{i+1} \psi(w) \Big)\ge
  \pi_{i,w}\Big( h(w) < \phi_{i,w}(w)+(c_0/2)\delta^{1-\eta} \Big)
\end{eqnarray}
and it follows from \cref{th:letrucdestrentederivees} (choosing $n=n(\eta)$ large enough) that
the latter probability is larger than $1-\delta^{18}$, as desired.

\section{Proof of \cref{thm:equilibrium_development,th:letrucdestrentederivees}}\label{sec:perturbation_continuum}

In this section, we use arguments from the theory of elliptic
PDEs to control the effect on the limit shape, i.e. on the solution of
\eqref{eq:EL}, of a small perturbation of its boundary conditions that
amounts to adding a linear tilt to the boundary datum. The estimates are a bit delicate because we need to control also the higher derivatives of the perturbed solution, uniformly with respect to the domains and the boundary conditions we consider.

Our first goal is \cref{thm:equilibrium_development}, that
is a description of the limit shape $\phi_{i,w}$ in the elliptical
domains $E_{i,w}$  of \cref{def:Eiw}, with boundary condition given by a perturbation of $\phi|_{\partial E_{i,w}}$, with  $\phi$
the limit shape of \cref{ass:ls}.  The first step is to rescale
the problem in order to work in a fixed domain, namely the disk
$B(0,1)$, instead of the ellipses $E_{i,w}$.
\begin{definition}
For $w \in W$,
let $t_w:\R^2\mapsto \R^2$ be the affine map
\begin{eqnarray}
  \label{eq:tw}
  t_w(x)=w+T_w x, \quad T_w=(-\nabla^2\psi(w))^{-1/2} 
  %% , \quad T_w= \left(  \begin{array}{ll}
  %%         1& 0\\
  %%         0 & \frac{m_2^2(w)}{m_1^2(w)}
  %%       \end{array}
  %% \right)
  ,
\end{eqnarray}
see \eqref{eq:gradpsi}.
 Note that $t_w(B(0, r_i) ) = E_{i,w}$ for all
 $i$, with $E_{i,w}$ the ellipse of Definition \ref{def:Eiw}. Given $w\in W$, $i\ge1$ 
 we also define rescaled versions $\Phi_{i,w}:B(0, 1)\mapsto \R$ of the limit shapes $\phi_{i,w}:E_{i,w}\mapsto \R$ of \cref{thm:equilibrium_development}
 by
 \begin{eqnarray}
   \label{eq:PhiIW}
\Phi_{i,w}(x) = \frac{1}{r_i}(\phi_{i,w}(  t_w(r_i x)) - \phi_{i,w}(w)).   
 \end{eqnarray}
\end{definition}
Note that the $\Phi_{i,w}$ has been normalized so that
$\Phi_{i,w}(0)=0$ and that, since $\phi_{i,w}$ solves the
Euler-Lagrange equation \eqref{eq:EL} in $E_{i,w}$, $\Phi_{i,w}$ is a
solution $u:B(0,1)\mapsto \R$ of the modified PDE
\[\label{eq:PDEtransformed}
\sum_{a,b=1}^2\sigma^{(w)}_{ab}(\nabla u(x))\partial^2_{x_a x_b}u(x)=0
\]
where, denoting $\Sigma$ and $\Sigma^{(w)}$ the $2\times 2$ matrices of elements $\sigma_{ab}$ and $\sigma^{(w)}_{ab}$ respectively, we have
\begin{eqnarray}
  \label{eq:sS}
  \Sigma^{(w)}(\cdot)=(T_w)^{-1} \Sigma((T_w)^{-1}\cdot)(T_w)^{-1}.
\end{eqnarray}
Note that, like \eqref{eq:EL}, Eq. \eqref{eq:PDEtransformed} is a non-linear elliptic PDE.

In analogy with the rescaled, perturbed
limit shapes $\Phi_{i,w}$ just defined, we introduce also the
rescaled, unperturbed limit shapes. Namely, with $\phi$ denoting the limit shape in
\cref{ass:ls}:
\begin{definition}
  \label{def:W}
Given $v_0>0$ we let
$\mathcal{W}$ be the following set of pairs
consisting of a point and a $C^\infty$ function from $B(0,1)$ to $\R$:
\[
\mathcal{W} = \{(w,  \Phi) : w \in \overline{W}, \Phi:B(0,1)\mapsto \R,    \Phi(\cdot) = \frac{1}{r}(\phi( t_w(r\cdot)) - \phi(w)) \text{ for some } r\in [0,v_0]\},
\]
where  for $r = 0$, by convention the equation for $\Phi$ denotes the linear map
\begin{eqnarray}
  \label{eq:linphi}
   \Phi(\cdot):=\langle  \nabla \phi(w) , t_w (\cdot) - w \rangle. 
\end{eqnarray}  
\end{definition}
The constant $v_0$ will be taken sufficiently small later, see
\cref{rem:r0} and the proof of Corollary \ref{cor:rog}.

\begin{remark}
    In the rest of this section, we use the notation
    $W^{k,p}(U)$ and $W^{k,p}_0(U)$ to denote the Sobolev spaces of
    functions with derivatives of order up to $k$ belonging to $L^p$
    in the domain $U$, and the subscript zero indicates that functions
    have ``zero boundary conditions'' at $\partial U$. When no confusion arises, we omit the argument $U$. For details on Sobolev spaces see
    \cite[Chapter 8]{krilov}, where these spaces are denoted $W^k_p$
    and ${\stackrel 0 {W^k_p}}$, respectively.
  
  In the following, whenever Sobolev spaces $W^{k,p}$ and
  Sobolev norms $\|\cdot\|_{k,p}$ appear, one should keep in mind
  that, in view of \cref{prop:fluctuations_reference}, we need the
  statements for $k=33, p=3$.  We write most statements for generic
  exponents $k,p>2$ in order to emphasize that the specific values
  just mentioned play no particular role in the proofs of this section (see also
  \cref{rem:embed}).
\end{remark}

Note that, given $r>0$ and $w$, the function $\Phi$ in \cref{def:W} is
simply the limit shape $\phi$, restricted to an ellipse centered at
$w$ and of size $r$, up to an affine transformation of space that
turns its domain into the unit disk $B(0,1)$. As is the case for
$\Phi_{i,w}$, $\Phi$ is normalized to equal zero at the center of the
ball $B(0,1)$, and it satisfies the PDE \eqref{eq:PDEtransformed}.

\begin{lemma}\label{lem:compact}
	Consider the natural topology on $\mathcal W$ induced by the Euclidean distance in $\R^2$ for the coordinate $w$, and the $\|\cdot\|_{k,p}$ Sobolev norm on functions in $B(0,1)$ for the coordinate $\Phi$. With this topology, $\mathcal{W}$ is a compact set.
\end{lemma}
\begin{proof}
  By construction, $\mathcal{W}$ is the image of $\overline{W}\times [0,v_0]$ via the map
  \[(w,r) \to \big(w,\Phi),\quad \Phi(\cdot):=\frac{1}{r}(\phi( t_w(r\cdot)) - \phi(w))
    \]
    so it is enough to show that this function is continuous. For any
    $r > 0$ this is trivial. For the continuity for $r=0$, note that
    any derivative of order 2 or more of $ \Phi$ converges to
    $0$ as $r \to 0$ uniformly with respect to $w$ (because the limit shape $\phi$ is $C^\infty$ in $\overline W$, thanks to  \cref{ass:ls}), so that $\Phi$ converges
    to the linear map \eqref{eq:linphi} in $\|\cdot\|_{k,p}$
    topology uniformly as $r \to 0$. This linear map depends
    continuously on $w$ and this completes the proof.
  \end{proof}
  We need two more definitions before entering the heart of the proofs.
  
  \begin{definition}
    Given $v_0>0$, let 
 \[\mathcal{L}  = \{ \ell : \R^2 \to \R \text{ linear }, \norm{\ell} \leq v_0\}\] be the compact set of linear maps on $\R^2$ of norm at most $v_0$ (the choice of norm is arbitrary and unimportant for the following).
\end{definition}

\begin{definition} For any $(w, \Phi)\in \mathcal{W}$, we let $ F_{w,  \Phi}$ be the following map that takes as input a pair  $(\ell,f)\in\mathcal L\times C^2_0(B(0,1)) $ (with $C^2_0(B(0,1))$ the set of $C^2$ functions $f:B(0,1)\mapsto \mathbb R$ that vanish on $\partial B(0,1)$) and outputs the function $F_{w,\Phi}(\ell,f):B(0,1)\mapsto \R$:
  \begin{eqnarray}
    \label{riva}
	F_{w,\Phi}(\ell,f)(x)= \sum_{i,j=1}^2 \sigma^{(w)}_{ij}(\nabla(\Phi+\ell+f)(x))\partial^2_{x_i x_j}( \Phi+f)(x), \quad x\in B(0,1), 
	\end{eqnarray}
        with $\sigma^{(w)}(\cdot)$ defined as in \eqref{eq:PDEtransformed}.
      \end{definition}
      \begin{remark}
        \label{rem:r0}
      Recall that 
      $\sigma_{ab}$ is  well defined only provided that its argument is in the
      Newton polygon $\mathbb T$, i.e., $\sigma_{ab}^{(w)}(u)$ is well defined only when  $(T_w^{-1})u\in\mathbb T $, recall \eqref{eq:sS}. On the other hand, $\nabla \Phi (x)=T_w \nabla \phi(t_w(r x))$ and, by assumption \cref{ass:ls}, $\nabla\phi$ is in the Newton polygon,   at distance at least $a>0$ from its boundary. In the following, therefore, we assume that the constant $v_0$ in the definition of $\mathcal L$ and
      the sup-norm of $\nabla f$ on $B(0,1)$ are small enough (as a function of the constant $a$) so that $\sigma_{ab}^{(w)}$ is well-defined.
      \end{remark}
      \begin{remark}
      Note that, by construction, $F_{w,\Phi}(0,0)=0$ (the identically
      zero function on $B(0,1)$), because, as observed above, $\Phi$ is just a rescaled
      version of the limit shape and therefore solves the PDE \eqref{eq:PDEtransformed}. Also, if $0\le i\le i_{max}$, letting
      \[
      \label{eq:wayofre}\Phi(\cdot):=\frac1{r_i}(\phi(t_w(r_i\cdot))-\phi(w))\] and 
      \[\ell(\cdot):= \eps_i \langle \nabla \psi(w) , t_w(\cdot)-w \rangle,\label{ltw}\] then
      \[F_{w, \Phi}( \ell ,  \Phi_{i,w} - 
      \Phi-\ell) = 0,\]
    with $\Phi_{i,w}$ defined in \eqref{eq:PhiIW}. This is just a way of rewriting that $\Phi_{i,w}$ is
    a solution of \eqref{eq:PDEtransformed}.
\label{rem:ms}
      \end{remark}

      In the following, we will let the linear maps
      $d_1 F_{w,\Phi}(\ell,f),d_2F_{w,\Phi}(\ell,f)$ denote the
      differential maps of $F_{w,\Phi}$ with respect to the first or
      second argument, computed at $(\ell,f)$. When the differentials
      are computed at $(0,0)$, we write simply $d_1F_{w,\Phi}, d_2 F_{w,\Phi}$ for
      brevity.  A direct computation gives that $d_2 F_{w,\Phi}(\ell,0)$ is a linear
      elliptic operator with non-constant but smooth coefficients
      given by
\begin{multline}
\label{eq:d2F}
d_2 F_{w,\Phi}(\ell,0)\circ f=\sum_{i,j=1}^2\sigma^{(w)}_{ij}(\nabla(\Phi+\ell))\partial^2_{x_i x_j}f\\+\sum_{k=1}^2\partial_{x_k}f\sum_{i,j=1}^2\sigma^{(w)}_{ijk}(\nabla(\Phi+\ell))\partial^2_{x_i x_j}\Phi,
\end{multline}
with $\sigma^{(w)}_{ijk}$ the derivative of $\sigma^{(w)}_{ij}$ with respect to its $k^{th}$ argument.
As for  $d_1 F_{w,\Phi}$, it transforms a linear map $\ell$ into 
\begin{eqnarray}
\label{eq:d1F}
d_1 F_{w,\Phi}\circ \ell=\sum_{k=1}^2\ell_k\sum_{i,j=1}^2\;\sigma^{(w)}_{ijk}(\nabla\Phi)\partial^2_{x_i x_j}\Phi
, \quad \ell_k:=\partial_{x_k}\ell(x)\end{eqnarray}
which is a function from $B(0,1)$ to $\R$.

Let us now start the proof of
\cref{thm:equilibrium_development,th:letrucdestrentederivees} using
the above definitions. The first step is based on a classical result
on linear elliptic PDEs.

\begin{proposition}\label{claim:inversibilite} For every pair of  integers $k,p>2$,  for every
  $(w, \Phi)\in\mathcal{W}$ and every $\ell \in \mathcal{L}$,
the map  $f\mapsto d_2 F_{w,\Phi}(\ell,0)\circ f$ is a continuous invertible map from
  $W^{k, p}_0:=W^{k, p}_0(B(0,1))$ to $W^{k-2, p}:=W^{k-2, p}(B(0,1))$. Furthermore, there
  exists a constant $C$ such that for all
  $f \in W^{k, p}_0$,
  $ \norm{ f}_{{k, p}} \leq C \norm{d_2 F(\ell,0)\circ f}_{{k-2,
      p}}$. The constant $C$ is uniform over admissible choices
  $(w, \Phi)\in \mathcal W, \ell\in\mathcal L$.%
\end{proposition}

\begin{proof}[Proof of Proposition \ref{claim:inversibilite}]
  The map $f\mapsto d_2 F_{w,\Phi}(\ell,0)\circ f$ is trivially
  continuous, see \eqref{eq:d2F}, so the only question is the existence
  and continuity of the inverse map.  Since the disk $B(0,1)$ is a smooth
  bounded domain and the coefficients of the operator $d_2 F_{w,\Phi}(\ell,0)$
  are $C^\infty$, 
  \cite[Th. 11.3.2 and Th. 9.2.3]{krilov} imply that $d_2 F_{w,\Phi}(\ell, 0)$ is invertible and that
  $[d_2 F_{w,\Phi}(\ell,0)]^{-1}$ is continuous (and therefore bounded) from $W^{k,p}_0$ to
  $W^{k-2, p}$.
	
  For the uniform bound on the norm of
  $[d_2 F_{w,\Phi}(\ell,0)]^{-1}$, first note that the map
  $(w, \Phi,\ell) \to d_2 F_{w,\Phi}( \ell, 0)$ is continuous when
  considering the topology described in Lemma \ref{lem:compact} on the
  domain $\mathcal W\times\mathcal L$ and the operator norm from
  $W^{k, p}_0 $ to $W^{k-2,p}$ on the codomain. Since the inverse is
  defined everywhere, this implies that
  $(w, \Phi, \ell) \to [d_2F _{w,\Phi}( \ell, 0)]^{-1}$ is also
  continuous. We conclude because $\mathcal{W} \times \mathcal{L} $ is
  compact by Lemma \ref{lem:compact}.
\end{proof}

The main point in the proof of \cref{thm:equilibrium_development} is
the following implicit function theorem. For
$(w, \Phi) \in \mathcal{W}$ and $\ell \in \mathcal{L}$, let
\begin{eqnarray}
  \label{eq:chichi}
\chi = -(d_2F_{w, \Phi})^{-1} \circ d_1 F_{w, \Phi} \circ \ell,  
\end{eqnarray}
i.e $\chi\in W^{k,p}_0$ is the function $\chi:B(0,1)\mapsto \R$ which
solves the linear elliptic PDE
\begin{eqnarray}
  \label{eq:chichi2}
 d_2 F_{w, \Phi}\circ \chi = - d_1 F_{w, \Phi}\circ \ell .  
\end{eqnarray}
\begin{proposition}
  \label{prop:analysefonct} If the constant $v_0$ in the definition of
  $\mathcal W,\mathcal L$ is chosen small enough, then for
  every $(w, \Phi) \in \mathcal{W}$ and every $\ell \in \mathcal{L}$,
  there exists an unique $f\in W^{k,p}_0(B(0,1))$ such that
  $F_{w, \Phi}(\ell, f )=0$. Furthermore the map
  $(w, \Phi, \ell) \to f$ is continuous 
  and satisfies
	\[\label{eq:implicit}
	\norm{f -\chi}_{{k, p}} \leq C \norm{F_{w,\Phi}( \ell, \chi)}_{{k-2, p}}
	\]
	where the constant $C$ is uniform with respect to $(w, \Phi)\in\mathcal W$ and $\ell\in\mathcal L$.
      \end{proposition}
      \begin{remark}
        \label{rem:embed}
We recall also (see for instance \cite[Th. 10.4.10]{krilov}) that for integer $k\ge1 $ and $p>2$, 
$W^{k,p}(B(0,1))$ is continuously embedded into 
\[C^{k-1,1-2/p}(B(0,1))\subset W^{k-1,\infty}(B(0,1))\] where $C^{k-1,1-2/p}(B(0,1))$ is the H\"older space of functions whose derivatives up
to order $k-1$ are $(1-2/p)$-H\"older continuous.
In particular,
\begin{eqnarray}
  \label{eq:sobolev}
  \|u\|_{C^{k-1,1-2/p}(B(0,1))}\le N   \|u\|_{k,p}
\end{eqnarray}
for some constant $N=N(k,p)$.
Therefore, taking the integers $k,p>2$ we have that $f$ in \cref{prop:analysefonct} is actually twice continuously differentiable on $B(0,1)$ and therefore is a classical solution of the PDE $F_{w,\Phi}(\ell,f)=0$, see \eqref{riva}.
      \end{remark}

\begin{corollary}
  \label{cor:rog}
  The family $\{\Phi_{i,w}\}_{i\le i_{max},w\in W}$ is precompact in $W^{k,p}(B(0,1))$. \end{corollary}
\begin{proof}[Proof of Corollary \ref{cor:rog}]
  For this, recall \cref{rem:ms}. One has  $r_i\le
  r_{i_{max}}=\delta^{3\eta/2}\le v_0$, with $v_0$ the constant that
  enters the definition of $\mathcal L,\mathcal W$, so that the pair $(w,\Phi)$ belongs to $\mathcal W$. Similarly, 
  the linear map $\ell$  satisfies
  $\ell\in \mathcal L$ provided that the constant $\epsilon_0$
  introduced just before \eqref{eq:A} is small enough (because $v_0$
  in the definition of $\mathcal L$ is fixed, while $\epsilon_i\le
  \epsilon_0$.) The claim of the corollary then follows from \cref{claim:inversibilite} together with Eq.   \eqref{eq:wayofre}, which says that $f$ of \cref{claim:inversibilite} is in this case just $\Phi_{i,w}-\Phi-\ell$.
\end{proof}

\begin{proof}[Proof of \cref{prop:analysefonct}]
      In view of $F_{w,\Phi}(0,0)=0$ and of \eqref{eq:chichi2},
      \cref{prop:analysefonct} is almost the usual statement of the
      implicit function theorem applied to $F_{w,\Phi}$, but with a
      precise estimate on the error bound in \eqref{eq:implicit} which
is      a bit more delicate than usual and which will be important
later.

  First, existence and uniqueness in the statement comes from the
 corresponding statement for the limit shape PDE, equation
  \eqref{eq:EL}. Indeed, $u:=\Phi + f+\ell$ solves
  \eqref{eq:PDEtransformed} with Dirichlet-type boundary conditions on
  $\partial B(0,1)$, which is equivalent to solving \eqref{eq:EL} in
  an ellipse.  However, we will prove existence of $f$ constructively,
  by a fixed point argument; this will give as byproduct the
  claimed continuity and the estimate on the error bound.

  For the moment, let us fix $(\Phi, w)\in\mathcal W, \ell\in\mathcal
  L$.  Since the function $F_{w,\Phi}(\cdot, \cdot )$ is smooth with
  respect to both arguments it is easy to see that, given any constant
  $\mathcal C>0$, we can find $\eps$ small enough such that, for any
  two functions $g$ and $g'$ in $W^{k,p}_0$ with
  $\norm{g}_{k,p},\norm{g'}_{k,p} \leq \eps$, we have
\[\label{eq:boundDLF}
\norm{F_{w,\Phi}(\ell, g' ) - F_{w,\Phi}( \ell, g) - d_2 F_{w,\Phi}(
  \ell, g) \circ (g' - g) }_{{k-2, p}} \leq \mathcal C\norm{g -
  g'}_{{k, p}}.
\]
In particular, it turns out to be convenient to choose
        \begin{eqnarray}
          \label{eq:constC}
          \mathcal C:=\frac{1}{5 \sup_{(w,\Phi)\in\mathcal
              W}\norm{(d_2 F_{w,\Phi})^{-1}}}
        \end{eqnarray}
where here and later in this proof, $\norm{d_2 F_{w,\Phi}^{-1}}$
denotes the operator norm  of $(d_2
F_{w,\Phi})^{-1}$ from $W^{k-2,p}$ to $W^{k,p}_0$. Thanks to \cref{claim:inversibilite}, we have
$\mathcal C>0$ . By continuity of $d_2 F_{w,\Phi}(\cdot,\cdot)$ (the
continuity is uniform w.r.t. $(w,\Phi)\in\mathcal W$ because $\mathcal
W$ is compact by \cref{lem:compact}), for any $g\in W^{k,p}_0$ with
$\norm{g}_{k,p} \leq \eps$ with $\eps$ small enough, we also have
	\[\label{eq:bounddF-1}
          \norm{d_2 F_{w,\Phi}(\ell,g) - d_2 F_{w,\Phi}}
          \leq \mathcal C,
	  \]
          (here the norm is the operator norm from $W^{k,p}_0$ to
            $W^{k-2,p}$),
provided that $v_0$ in the definition of $\mathcal L$ is small enough,
so that the linear map $\ell$ has small norm. Once more, we 
choose $\mathcal C$ as in \eqref{eq:constC}.

We define $g_0 := \chi$ and, by induction,
\begin{equation}\label{eq:gngn-1}
g_{n+1} - g_{n} := -(d_2F_{w,\Phi})^{-1} \circ F_{w,\Phi}(
\ell, g_n),
\end{equation}
i.e the sequence $(g_{n})_{n \geq 0}$ is a Newton
approximation sequence, except that we keep constant the point where
the differential $d_2$ is computed. We will show that $g_n$ converges
to the desired solution $f$, which in addition satisfies the desired estimates.

To this purpose, we will prove by induction that
\begin{align}
\label{prima}\norm{F_{w,\Phi}( \ell, g_n)}_{{k-2, p}} \leq 2^{-n}
\norm{F_{w,\Phi}( \ell, g_0)}_{{k-2, p}}\\\label{seconda}
\norm{g_{n+1} -
  g_n}_{{k, p}} \leq 2^{-n} \norm{ (d_2F_{w,\Phi})^{-1} } \norm{F_{w,\Phi}( \ell,
 g_0)}_{{k-2, p}}
\end{align}
for all $n$ and, along the way, we will make sure that
\begin{equation}
  \label{terza}
\|g_n\|_{k,p}\le \epsilon,
\end{equation}
so that \eqref{eq:boundDLF} and \eqref{eq:bounddF-1} can be applied.
If this holds, then $\{g_n\}$ is a Cauchy sequence and the limit $f$ satisfies
$F_{w,\Phi}(\ell,f)=0$ and \eqref{eq:implicit} with $C=2\sup_{\mathcal W}\|(d_2 F_{w,\Phi})^{-1}\|$ which is finite
because $(w,\Phi)\mapsto (d_2 F_{w,\Phi})^{-1}$ is continuous and $\mathcal W$ is compact.

For $n=0$, \eqref{prima} and \eqref{seconda} are trivial. Also, we choose $v_0$ small enough in the definition of $\mathcal W,\mathcal L$ so that
\begin{eqnarray}
  \label{piccolo}
\|g_0\|_{k,p}\le \frac\epsilon2, \quad \|F_{w,\Phi}(\ell,g_0)\|_{k-2,p}\, \|(d_2 F_{w,\Phi})^{-1}\|\le \frac\epsilon4
\end{eqnarray}
so that in particular also \eqref{terza} holds for $n=0$.
Assume that the three claims hold up to some step $n-1$.
From  \eqref{piccolo} and \eqref{seconda} (for all $k\le n-1$) one easily sees that
\[
\|g_n\|_{k,p}\le \frac\epsilon2+2\|(d_2 F_{w,\Phi})^{-1}\|\, \|F_{w,\Phi}(\ell,g_0)\|_{k-2,p}\le \epsilon
\]
so that \eqref{terza} follows at step $n$.

As for \eqref{prima}, we write
\begin{multline}
  \|F_{w,\Phi}(\ell,g_n)\|_{k-2,p}\\\le \|F_{w,\Phi}(\ell,g_n)-F_{w,\Phi}(\ell,g_{n-1})-d_2 F_{w,\Phi}(\ell,g_{n-1})\circ(g_n-g_{n-1})\|_{k-2,p}\\+
  \|F_{w,\Phi}(\ell,g_{n-1})+d_2 F_{w,\Phi}\circ(g_n-g_{n-1})\|_{k-2,p}\\+
  \|(d_2 F_{w,\Phi}(\ell,g_{n-1})-d_2F_{w,\Phi})\circ(g_n-g_{n-1})\|_{k-2,p}.
\end{multline}
The second term in the last expression is zero by \eqref{eq:gngn-1}.
The sum of the first and third terms is upper bounded, thanks to \eqref{eq:boundDLF} and \eqref{eq:bounddF-1}, by
\begin{eqnarray}
  2\mathcal C \|g_n-g_{n-1}\|_{k,p}\le 4 \mathcal C \, 2^{-n}\|(d_2 F_{w,\Phi})^{-1}\|
  \|F_{w,\Phi}(\ell,g_0)\|_{k-2,p}\\\le  2^{-n}
  \|F_{w,\Phi}(\ell,g_0)\|_{k-2,p},
\end{eqnarray}
where in the first inequality we used \eqref{seconda} for $n-1$ and in the second one the definition of $\mathcal C$. Then, \eqref{prima} at step $n$ follows.
Given this, \eqref{seconda} at step $n$ is proven immediately.

Concerning the continuity statement, first observe that the map \[(w,
\Phi, \ell, g) \to F_{w, \Phi}(\ell, g)\] is continuous for the $W^{k,
  p}$ topology on $\Phi$ and $g$ and the $W^{k-2, p} $ topology on the
target space. Indeed, the map $(\Phi, \ell, g) \to \nabla (\Phi + \ell
+ g)$ is by definition continuous for the $W^{k-1, p}$ topology and in
particular (by \cref{rem:embed}) also for the $W^{k-2, \infty}$ topology. The maps
\[(u\in \mathbb T , w\in W) \to \sigma_{ij}^{(w)}(u), 1\le i,j\le 2\] are smooth and bounded if $T_w^{-1}u$ is uniformly bounded away from the boundary of the
Newton polygon $\mathbb T$, which is the case if $u=\nabla(\Phi+\ell+g)$ because the limit shape satisfies
\cref{ass:ls} and $\ell,g$ are small in the appropriate norms.  By
composition,  \[(w, \Phi, \ell, g) \to \sigma_{ij}^{(w)}(\nabla
(\Phi+\ell+g))\] is continuous
in $W^{k-2, \infty}$. Overall $F_{w, \Phi}(\ell, g)$ is a product of functions continuous in 
$W^{k-2, \infty}$ and functions continuous in  $W^{k-2, p}$, which is still continuous in
$W^{k-2, p}$. This concludes the continuity of $(w,
\Phi, \ell, g) \to F_{w, \Phi}(\ell, g)$.  A similar statement
holds for $d_1 F_{w, \Phi}$ and $d_2 F_{w, \Phi}$.
Recalling that the $g_n$ are defined through compositions of $F_{w,\Phi}$,
$d_2F_{w,\Phi}$ and $d_1 F_{w,\Phi}$ (the latter enters the definition of
$g_0$), we conclude that the functions $g_n$ are  continuous in terms of $(w,
\Phi, \ell)$. Together with the uniform control \[\norm{f - g_n }_{k,p}
\leq C\,2^{-n} \norm{F_{w,\Phi}(\ell, \chi)}_{k-2,p}\le C' 2^{-n},\]
this implies that $f$ is also a continuous function of $(w, \Phi,
\ell)$.
\end{proof}

We can now conclude the proof of our equilibrium estimates.
\begin{proof}[Proof of \cref{th:letrucdestrentederivees}]
  Our aim is to prove \eqref{eq:letruc} about the equilibrium measure
  $\pi_{i,w}$ in the discrete ellipse $(E^{1/2}_{i,w})^\d$. First of
  all, in order to apply \cref{cor:33}, it is convenient to work on a
  domain with diameter of order $1$ centered at zero, rather than with
  diameter of order $r_i$ and centered at $w$. This can be achieved
  via a trivial translation that maps $w$ to the origin and a
  rescaling of lengths by a factor $1/r_i$.  This way,
  $(E^{1/2}_{i,w})^\d$ is replaced by the discretization $(\mathcal
  E^{1/2}_{i,w})^{\d}$ (with lattice mesh $\delta'=\delta/r_i$ instead
  of $\delta$) of an ellipse $\mathcal E_{i,w}^{1/2}$ centered at
  zero, with horizontal axis of length $1/(2m_1(w))$ and of the same aspect
  ratio as $E^{1/2}_{i,w}$. At the same time, we rescale the height
  function (both the discrete one, $h$, and the limit shape,
  $\phi_{i,w}$) by multiplying it by $1/r_i$, and by setting it to 
  zero at the origin. That is, calling $\check h$ (resp.  $\check
  \phi_{i,w}$) the height functions thus obtained, we have
  \begin{eqnarray}
    \check h(x)=\frac1{r_i}\big(h(w+xr_i)-h(w)\big), \quad x\in (\mathcal E^{1/2}_{i,w})^\d,
  \end{eqnarray}
  and similarly for $\check \phi_{i,w}$. Note that the discrete
  gradients of $\check h$ between neighboring vertices of
  $(\mathcal E^{1/2}_{i,w})^{\d}$ are of order $\delta'$ instead of
  $\delta$, and that $\check \phi_{i,w}$ is a limit shape (i.e. a
  solution of \eqref{eq:EL}) in $\mathcal E^{1/2}_{i,w}$. We  call $\check\pi_{i,w}$ the uniform measure
  on height functions $\check h$ in $(\mathcal E^{1/2}_{i,w})^{\d}$,
  with boundary condition given as in \cref{piiw}, up to the rescaling
  just introduced.

  After this trivial rescaling, the inequality \eqref{eq:letruc} to be proven becomes
  \begin{eqnarray}
    \label{rog}
r_i^n\,\check \pi_{i,w}\Big[\big|\check h(x)-\check\phi_{i,w}(x)\big|^n\Big]\le C_n \delta^n|\log \delta|^n.
  \end{eqnarray}
  Since (recall  \cref{def:phiiw}) the limit shape $\check \phi_{i,w}$ is defined on an ellipse $\mathcal E_{i,w}$ that is   $\mathcal E^{1/2}_{i,w}$ expanded by a factor $2$, so that in particular $\mathcal E_{i,w}$ contains the closure of $\mathcal E^{1/2}_{i,w}$,
  applying  \cref{cor:33} we see that the l.h.s. of \eqref{rog} is upper bounded by
  \begin{eqnarray}
r_i^n C_n (\delta')^n |\log \delta'|^n= C_n \delta^n  |\log \delta'|^n.
  \end{eqnarray}
Next we remark that, in view of \eqref{eq:r0rimax}, one has
  \begin{eqnarray}
\delta^{1-3\eta/2}  \le  \delta'=\frac{\delta}{r_i}\le \delta^{(1+\eta)/2}, \quad i\le i_{max},
  \end{eqnarray}
  so that \eqref{rog} holds, up to changing the definition of $C_n$ by
  a multiplicative factor that depends only on $\eta$ and $n$. However,  as pointed out just after the statement of \cref{th:letrucdestrentederivees}, there is still an
  important  and delicate point to be checked, that is that the constants $C_n$ can be chosen
  uniform with respect to $i\le i_{max}$ and $w\in  W$. 
For this, we use a compactness argument based on \cref{claim:inversibilite}.

From \cref{th:letrucdestrentederivees} we know that, if
$\check\phi_{i,w}$ can be extended to some open domain $U$ with
$U\supset \overline{\mathcal E^{1/2}_{i,w}}$, then the constants $C_n$
can be chosen as a function $c_n(U,\check \phi_{i,w}|_U)$ and that,
given $U$, the map $f\mapsto c_n(U,f)$ is continuous with respect to
the Sobolev norm $W^{k,p}(U)$.  A first observation is that, while by
construction the natural domain of definition of $\check \phi_{i,w}$
is the ellipse $\mathcal E_{i,w}$, we are not forced to  take 
$U= \mathcal E_{i,w}$ or any other choice of $U$ that is different for each $(i,w)$; actually a finite number of such domains $U$
suffices. In fact, since the
aspect ratio of the ellipses is uniformly bounded w.r.t. $i\le i_{max},w\in W$, we can
find an integer $k$ and a collection $\{U_s\}_{s\le k}$ of open
domains of $\R^2$ such that for every $i\le i_{max},w\in W$,
  \[
\mathcal E^{1/2}_{i,w}\subset U_{s(i,w)} \subset \mathcal E_{i,w}
  \]
  for some $s(i,w)\le k$. Therefore, we can take $C_n$ as some
  function \[C_n=c'_n(s(i,w),\check\phi_{i,w}|_{U_{s(i,w)}}),\]
  continuous in its second argument. Since $s(i,w)$ takes finitely
  many values, for the issue of the uniform bound on the constants we
  can disregard the dependence of $c'_n$ on the argument $s$.
  Secondly, note that the functions $\check\phi_{i,w}:\mathcal
  E_{i,w}\mapsto \R$ and $\Phi_{i,w}:B(0,1)\mapsto \R$ of
  \eqref{eq:PhiIW} are immediately related by
  \[
\check \phi_{i,w}(T_w x) = \Phi_{i,w}( x), \quad T_w:=(-\nabla^2\psi(w))^{-1/2}
\]
see \eqref{eq:tw}. Since both $T_w$ and its inverse are bounded
uniformly in $w\in W$, and the restriction map
$\check\phi_{i,w}|_{\mathcal E_{i,w}}\mapsto
\check\phi_{i,w}|_{U_{s(i,w)}}$ is obviously a continuous map, we can
take $C_n$ as $C_n=c''_n(\Phi_{i,w})$, with $f\mapsto c''_n(f)$
continuous. Finally, the family $\{\Phi_{i,w}\}_{i,w}$ is
precompact in $W^{k,p}(B(0,1))$ by \cref{cor:rog}. This implies that
the supremum over $i,w$ of $c''_n(\Phi_{i,w})$ is finite.
\end{proof}

\begin{proof}[Proof of \cref{thm:equilibrium_development}]
  In the course of this  proof, 
         $C$ denotes a constant whose value can change from line to
        line.
  We work with $\Phi_{i,w}$, which is obtained from $\phi_{i,w}$ via the rescaling \eqref{eq:PhiIW} and at the end of the proof we go back to $\phi_{i,w}$ to prove \eqref{eq:sa3}.
  Fix $i\le i_{max}$ and $w\in W$ and, with the notations of Remark \ref{rem:ms}, let $\Phi(\cdot):B(0,1)\mapsto
\R$ be the rescaled limit
shape around $w$ and $\ell$ the linear map defined in \eqref{ltw}. If 
        $\chi$ is defined as in \eqref{eq:chichi}, by \cref{prop:analysefonct} and the
        definition of $\Phi_{i,w}$, we have
	\[
	\norm{\Phi_{i,w} -(\Phi + \ell  +  \chi) }_{{k, p}} \leq C \norm{F_{w,\Phi}( \ell,  \chi)}_{{k-2, p}}.
        \label{cucina}
	\]
        This should be read as saying that $\Phi_{i,w}=\Phi+\ell+\chi$
        plus an error term. We will first prove that $\chi$ is essentially a quadratic function and then show that
        $\norm{F_{w,\Phi}( \ell, \chi)}_{k-2, p}$ is very small.

        Recall that the function $\chi$ solves the linear equation
\begin{eqnarray}
 d_2 F_{w, \Phi}\circ \chi = - d_1 F_{w, \Phi}\circ \ell,
\end{eqnarray}
        where $d_2F_{w,\Phi}$ and $d_1F_{w,\Phi}$ are given by
        \eqref{eq:d2F} and \eqref{eq:d1F}, so that
        the latter equation can be written as
        \begin{eqnarray}
          \label{eq:chi}
          \sum_{r,s=1}^2(\partial^2_{x_r x_s}\Phi)\sum_{k=1}^2(\ell_k+\partial_{x_k}\chi)\sigma^{(w)}_{rsk}(\nabla\Phi)+\sum_{r,s=1}^2\sigma^{(w)}_{rs}(\nabla\Phi)\partial^2_{x_r x_s}\chi=0.
        \end{eqnarray}
        Because of the rescaling implicit in the definition of $\Phi$,
        one has $\norm{\partial^2_{x_r x_s} \Phi}_{{k-2, p}} \leq C
        r_i$ and the coefficients $\ell_k$ of the linear map $\ell $
        are of order $\eps_i$, so altogether $\norm{d_1 F_{w,\Phi}
          \circ \ell}_{{k-2, p}} \leq C \eps_i r_i$. Since $[d_2
          F_{w,\Phi}]^{-1}$ is a bounded linear map, we get
        $\norm{\chi}_{{k, p}} \leq C \eps_i r_i$. Thanks to
        \cref{rem:embed}, the $L^\infty$ norm of the second
        derivatives of $\chi$ is also bounded by $C \eps_i r_i$.

In order to show that $\chi$ is well approximated by a quadratic
function vanishing at the boundary of $B(0,1)$, let $Q(x) := (1 -
\norm{x}^2)$ (the additive constant is there so that $Q$ indeed vanishes on $\partial B(0,1)$). For any constant $b$, we view $\chi - b\, Q$ as the
solution of the equation
	\begin{multline}\label{eq:chi-aQ}
          d_2 F_{w,\Phi}\circ(\chi - b\, Q)(x) = -[d_1 F_{w,\Phi}\circ \ell](x)
          %%\sum_{i,j, k=1}^2\partial^2_{x_i x_j} \Phi \cdot  \ell_k\sigma^w_{ijk}(\nabla\Phi)
\\ -2b\sum_{k=1}^2\sigma^{(w)}_{kk}(\nabla\Phi)  - 2b\sum_{r,s,k=1}^2x_k\sigma^{(w)}_{rsk}(\nabla\Phi)\partial^2_{x_r x_s} \Phi.
\end{multline}
Since the surface tension $\sigma$ is strictly convex, or more precisely
$\sigma^{(w)}_{kk}>0$,  we can choose $b$ so that the right hand
side vanishes at the point $x=0$. Reasoning as for the above bound
on $\norm{\chi}_{k,p}$, we see that this choice satisfies $\abs{b}
\leq C\epsilon_ir_i$. Next, we show that, with this choice of $b$,
\[
\label{sa3}
\text{\big|r.h.s. of \eqref{eq:chi-aQ}}\big|\le C\epsilon_i r_i^2,
\text{ uniformly on } B(0,1).\] In fact, note that the space
derivatives of the right hand side of \eqref{eq:chi-aQ} are linear
combinations of terms of the form $\partial^3 \Phi\cdot (\ell + b x)
\sigma( \nabla \Phi)$, $\partial^2 \Phi \cdot (\ell + b x)
\partial(\sigma( \nabla \Phi))$, $b\sigma(\nabla\Phi)\partial^2\Phi$
and $b\partial (\sigma( \nabla \Phi))$, where we omitted all indices
to simplify notations. Thanks to the higher order in the derivatives
and to the estimate on $b$, all these terms are of order at most
$\eps_i r_i^2$.  Since the r.h.s. of \eqref{eq:chi-aQ} is zero for
$x=0$ by the choice of $b$, \eqref{sa3} holds in the entire disk
$B(0,1)$. Higher derivatives of the r.h.s. of \eqref{eq:chi-aQ} are
also bounded by $C \epsilon_i r_i^2$ or even smaller, since each
derivative acting on $\Phi$ brings a factor $r_i$:
  \[
  \label{higherderiv}
  |\partial^{n}_{x_{i_1}\dots x_{i_n}}\Phi|\le C_n r_i^{n-1}.
  \]
In conclusion,
\[
\|        d_2 F_{w,\Phi}\circ(\chi - b\,Q) \|_{k-2,p}\le C \epsilon_i r_i^2
\]
and, by \cref{claim:inversibilite},
\[
\label{frigo}
\|\chi - b\,Q\|_{k,p}\le C \epsilon_i r_i^2.
\]
	
	We now turn to the ``error term''
	\[
	F_{w,\Phi}( \ell ,\chi ) = \sum_{r,s=1}^2 \sigma_{rs}^{(w)}( \nabla ({\Phi} + \ell + {\chi})) \partial^2_{rs} ( \Phi + \chi).
	\]
	Comparing this equation to \eqref{eq:chi} and recalling that $F_{w,\Phi}(0,0)=0$, we see that
        \begin{multline}
        F_{w,\Phi}(\ell,\chi)=        \sum_{r,s,k=1}^2(\ell_k+\partial_{x_k}\chi)\sigma^{(w)}_{rsk}(\nabla\Phi) \partial^2_{x_r x_s}\chi \\+
\sum_{r,s=1}^2\partial^2_{x_r x_s}(\Phi+\chi)\Big[
        \sigma^{(w)}_{rs}(\nabla ({\Phi} + \ell + {\chi}))-\sigma^{(w)}_{rs}(\nabla\Phi)-\sum_{k=1}^2\sigma^{(w)}_{rsk}(\nabla\Phi)(\ell_k+\partial_{x_k}\chi)\Big]
        \end{multline} and note that the last expression contains  a second order Taylor expansions  of  $\sigma_{rs}^{(w)}$. Using the apriori bound $\norm{ \chi}_{k, p} \leq C \eps_i r_i$, $\|\ell\|\le C\epsilon_i$ and the uniform bound \eqref{higherderiv}, we see that \[\norm{F_{w,\Phi}( \ell ,\chi )}_{k-2, p} \leq  C \eps_i^2 r_i\label{freezer}.\]
	Altogether, putting together \eqref{cucina}, \eqref{frigo} and \eqref{freezer}
        we have shown that
        \begin{eqnarray}
          \label{ottperPhi}
\|\Phi_{i,w}-(\Phi+\ell+b\,Q)\|_{k,p}\le C(\epsilon_i r_i^2+\epsilon_i^2 r_i)
        \end{eqnarray}
if the constant $b$ is precisely chosen as above.
        
	Finally we need to undo the scaling and show that
        \eqref{ottperPhi} indeed gives \eqref{eq:sa3} and the other
        statements of the Theorem.
        We have for $y\in E_{i,w}$, up to an additive constant,
        \begin{multline}
          \phi_{i,w}(y)=r_i\Phi_{i,w}\big(\frac1{r_i}t_w^{-1}(y)\big)=\phi(y)+\ell(t_w^{-1}(y))-\frac b{r_i} \|t_w^{-1} (y)\|^2+R(y)\\
          R(y)={r_i}\Big[\Phi_{i,w}-(\Phi+\ell+b\, Q)\Big]\Big(\frac1{r_i}t^{-1}_w(y)\Big).
        \end{multline}
        By construction, one has
        \[
\frac b{r_i} \|t_w^{-1} (y)\|^2=a\epsilon_i Q_w(y),\quad a:=\frac b{\epsilon_i r_i}
\]
with $Q_w$ the quadratic form defined in \eqref{eq:quadform}
and, from \eqref{ltw}, we see that
\[
\ell(t_w^{-1}(y))=\epsilon_i \langle\nabla\psi(w),y-w\rangle.
\]
Note also that  $R$ vanishes on the boundary of $E_{i,w}$ since  $Q|_{\partial B(0,1)}=0$  while the boundary condition of $\Phi_{i,w}$  and $\Phi+\ell$ are the same.
Also,  by \eqref{ottperPhi}, the first  derivatives  of $R$ are bounded by  $C(\epsilon_i r_i^2+\epsilon_i^2 r_i)$. Since the ellipse $E_{i,w}$ has diameter of order $r_i$, we conclude that $|R(x)|\le C
(\eps_i^2 r_i^2+ \eps_i r_i^3)$. Altogether, up to an additive constant, 
	\[
	\phi_{i,w} (y)= \phi(y)+\epsilon_i \langle \nabla\psi(w), y-w\rangle -a\epsilon_i Q_w (y)+ O(\eps_i^2 r_i^2+ \eps_i r_i^3) 
	\]
	where the $O(\eps_i^2 r_i^2+ \eps_i r_i^3) $ term is uniform in $y\in E_{i,w}$. By \eqref{eq:ri} and $r_i\le \delta^{3
          \eta/2}$, we see that  $\epsilon_i r_i^3=o(\delta)$ uniformly in $i$, which completes the proof of \eqref{eq:sa3}.
 The claimed uniformity of the error terms comes from the uniform bound on  the constant
in \cref{prop:analysefonct}  and on all operator
norms involved, which are continuous functions over the compact set
$\mathcal{W} \times \mathcal{L}$.

The validity of \eqref{eq:C''} is simply due to the fact that
$\phi_{i,w}$ has to satisfy the correct boundary condition on
$\partial E_{i,w}$, which imposes uniquely the additive constant
$\epsilon_i C'_{i,w}$ in \eqref{eq:sa3}.

Finally, to see that $a=b/(\epsilon_i r_i)$ can be taken smaller than $1/4$ by taking $\xi$ small enough, let us go back to the definition of $b$ as the value such that the r.h.s. of \eqref{eq:chi-aQ} computed for $x=0$ vanishes.
Recalling \eqref{eq:PhiIW} and \eqref{eq:sS}, this can be rewritten as
\begin{multline}\label{eq:ato0}
  2a \sum_{k=1}^2 \partial^2_{x_k}\psi(w)\sigma_{kk}(\nabla \phi_{i,w}(w))=
  -\epsilon_i r_i \sum_{u,v,k=1}^2 \sigma_{uvk}(\nabla\phi_{i,w}(w))\partial_{x_k}\psi(w)\partial^2_{x_u x_v}\phi_{i,w}(w), 
\end{multline}
with $\sigma_{uvk}$ the derivative of $\sigma_{uv}$ with respect to its $k^{th}$ argument.
Using \eqref{eq:confrontoderivate} and the fact that $\sigma_{kk}$ is strictly positive (by strict convexity of the surface tension), one sees that
$a\le C\xi\le 1/4$ for $\xi$ small enough.

\end{proof}

\section*{Acknowledgements}
F. T. gratefully acknowledges financial support of the  Austria Science Fund (FWF), Project Number P 35428-N. Both authors acknowledge financial support  of Agence Nationale de la Recherche (ANR), Project number ANR-18-CE40-0033.

\end{document}